\documentclass{article}
\usepackage{eurosym}
\usepackage[T1]{fontenc}
\usepackage{float,color,ulem}
\usepackage{makeidx}
\usepackage{amssymb}
\usepackage{amsfonts}
\usepackage{amsmath}
\usepackage{hyperref}
\usepackage{caption}
\usepackage{graphicx,multirow}
\usepackage{epsf,epsfig,subfigure,diagbox}
\usepackage[toc,page,title,titletoc,header]{appendix}
\usepackage{indentfirst}
\hypersetup{urlcolor=blue,linkcolor=blue ,citecolor=blue,colorlinks=true}
\usepackage[english]{babel}
\usepackage{bbm, dsfont}
\newtheorem{theorem}{Theorem}[section]

\newtheorem{assumption}[theorem]{Assumption}

\newtheorem{definition}[theorem]{Definition}
\newtheorem{example}[theorem]{Example}

\newtheorem{lemma}[theorem]{Lemma}

\newtheorem{proposition}[theorem]{Proposition}

\newtheorem{remark}[theorem]{Remark}

\numberwithin{equation}{section}
\newenvironment{proof}[1][Proof]{\textbf{#1.} }{\ \rule{0.5em}{0.5em}}

\newcommand{\G}{{\rm G}}

\newcommand{\D}{{\rm D}}
\newcommand{\F}{{\rm F}}
\newcommand{\Li}{{\rm L}}
\newcommand{\M}{{\rm M}}
\newcommand{\C}{{\rm C}}
\newcommand{\V}{{\rm V}}
\newcommand{\W}{{\rm W}}
\newcommand{\loc}{{\rm loc}}

\begin{document}

\title{\textbf{{Monotone abstract non-densely defined Cauchy problems applied to age structured population dynamic models}}}

\author{Pierre Magal$^{1}$ Ousmane Seydi$^{2}$ and Feng-Bin Wang$^{3,4}$ \\
{\small $^{1}$ \textit{Univ. Bordeaux, IMB, UMR 5251, F-33400 Talence, France} } \\
{\small \textit{CNRS, IMB, UMR 5251, F-33400 Talence, France.}} \\
{\small $^{2}$ \textit{D\'epartement Tronc Commun,\'Ecole Polytechnique de Thi\`es,
S\'en\'egal}} \\
\small{$^{3}$ Department of Natural Science in the Center for General Education,}\\
\small{ Chang Gung University, Guishan, Taoyuan 333, Taiwan.} \\
\small{$^{4}$ Community Medicine Research Center,}\\
\small{Chang Gung Memorial Hospital, Keelung, Keelung 204, Taiwan.}
}

\maketitle
\begin{abstract}
In this article we first derive some sufficient conditions to establish the monotonicity and comparison principles of the semi-flow generated by non-densely defined Cauchy problems. We apply our results to a class of age structured population models. As a consequence we obtain a monotone semi-flow theory and some comparison principles for age structured models.
\end{abstract}
\vspace{0.2in}\noindent \textbf{Key words}. Semilinear differential equations, non-dense domain, integrated semigroup, monotone semiflow, comparison principle,  age structured population dynamics models.

\vspace{0.1in}\noindent \textbf{AMS Subject Classification}. 37N25, 92D25, 92D30

\tableofcontents

\section{Introduction}

In this article we consider an abstract semi-linear Cauchy problem
\begin{equation}\label{1.1}
\dfrac{du(t)}{dt}=Au(t)+F(t,u(t)), \text{ for } t \geq 0, \text{ with } u(0)=u_0 \in \overline{D(A)},
\end{equation}
where $A:D(A) \subset X \to X$ is a linear operator on a Banach space $X$, and $F:[0, \infty) \times \overline{D(A)} \to X$ is continuous. We assume that the map $x \to F(t,x)$ is Lipschitz on the bounded sets of $\overline{D(A)}$ uniformly with respect to $t$ in a bounded interval of $[0, \infty)$. We point out that $D(A)$ is not necessarily dense in $X$ and $A$ is not necessarily a Hille-Yosida operator.

Population dynamics is one of the interesting subjects in mathematical biology, and a central aim is to study its long-term behavior of the associated models.
Monotonicity methods and comparison principles are the main tools in the investigation of the global dynamics of such model systems. In the existing works, monotonicity methods and comparison arguments with dynamical system approaches have been well developed in ordinary differential equations, delay differential equations, and partial differential equations.
We refer Smith \cite{Smith95}, Hirsch ans Smith \cite{Hirsch-Smith} and Zhao \cite{Zhao2017} for more results and references on this subject. However, very few studies have addressed monotonicity and comparison for system \eqref{1.1}. The existence of the semiflow as well as the positivity of the solutions of \eqref{1.1} has been addressed by Magal and Ruan \cite{Magal-Ruan09a, Magal-Ruan2018}. The condition used in \cite{Magal-Ruan09a, Magal-Ruan2018} is a special case of the so called subtangential condition. In this paper, we intend to further extend this type of analysis to get a comparison principle between non negative solutions of \eqref{1.1}. Thus, we will derive the theory of monotone semiflow, the comparison principle and invariance of solutions for system \eqref{1.1}. Several examples of differential equations, such as delay differential equations \cite{Ducrot-Magal-Ruan2013, Liu-Magal-Ruan}, parabolic equation with non-linear and non local boundary conditions \cite{Chu-Ducrot-Magal-Ruan, Ducrot-Magal-Prevost} can be put into the present framework of non densely abstract Cauchy problem \eqref{1.1}. More examples can be found in \cite{Magal-Ruan2018}. Thus, our developed results in this paper will have a wide range of applicability.

The paper is organized as follows. In sections 2 and 3 we recall some basic results about non densely defined Cauchy problems. In section 4 we prove a result on the monotonicity of the semiflow. Section 5 is devoted to the establishment of  the comparison principle. In the last section, we provide applications to age structured population dynamics models and we refer to the book of Webb \cite{Webb-85book} and Magal and Ruan \cite{Magal-Ruan2018} for more results on this topic.

\section{Preliminary results}
Let $A: D(A) \subset X \to X$ be a linear operator. In the following we use the following notations
$$
X_0:=\overline{D(A)}
$$
and $A_0:D(A_0) \subset X_0 \to X_0$ the part of $A$ in $X_0$ that is
$$
A_0 x=A x, \quad \forall x\in D(A_0),
$$
and
$$
D(A_0)=\lbrace x\in D(A): Ax\in X_0\rbrace.
$$
\begin{assumption} \label{ASS2.2}
We assume that
\begin{itemize}
\item[{\rm (i)}] There exist two constants $\omega_A\in \mathbb{R}$ and $M_A\geq 1$,
such that $(\omega_A,+\infty)\subset \rho (A)$ and
\begin{equation*}
\left\Vert (\lambda I-A)^{-k}\right\Vert _{\mathcal{L}(X_0)}\leq
M_A \left( \lambda -\omega_A \right) ^{-k},\;\forall \lambda >\omega_A,\;k\geq 1.
\end{equation*}
\item[{\rm(ii)}] $\lim_{\lambda \rightarrow +\infty }(\lambda I-A)^{-1}x=0,\
\forall x\in X$.
\end{itemize}
\end{assumption}
It is important to note that Assumption \ref{ASS2.2} does not say that $A$ is a Hille-Yosida linear operator since the operator norm in \textit{i)} is taken into $X_0 \subseteq  X$ (where the inclusion can be strict) instead of $X$. Further, it follows from \cite{Magal-Ruan2018} that $\rho(A)=\rho(A_0)$. Therefore by Assumption \ref{ASS2.2}, $(A_0,D(A_0))$ is a Hille-Yosida linear operator of type $(\omega_A,M_A)$ and generates a strongly continuous semigroup $\lbrace T_{A_0}(t) \rbrace_{t\geq 0}\subset \mathcal{L}(X_0)$ with
\begin{equation*}
\Vert T_{A_0}(t)\Vert_{\mathcal{L}(X_0)}\leq M_A e^{\omega_A t}, \quad \forall t\geq 0.
\end{equation*}
As a consequence
$$
\lim_{\lambda \to + \infty} \lambda \left(\lambda I -A \right)^{-1}x =x
$$
only for $x \in X_0$. It is certainly worth pointing out that the above limit does not exist in general whenever $x$ belongs to $X$.

We summarize the above discussions as follows.
\begin{lemma}\label{lemma2.2} \ Assumption \ref{ASS2.2} is satisfied if and only if
there exist two constants, $M_A\geq 1$ and $\omega_A \in \mathbb{R},$
such that $\left( \omega_A ,+\infty \right) \subset \rho (A)$ and
$A_{0}$ is the infinitesimal generator of a $C_{0}$-semigroup
$\left\{ T_{A_{0}}(t)\right\} _{t\geq 0}$
on $X_{0}$ which satisfies $\left\| T_{A_{0}}(t)\right\| _{\mathcal{L}%
(X_{0})}\leq M_Ae^{\omega_A t},\forall t\geq 0$.
\end{lemma}

Next, we consider the non  homogeneous Cauchy problem
\begin{equation}\label{2.1}
v^\prime(t)=\ A v(t)+f(t), \quad t\geq0 \quad \text{and} \quad v(0)=v_0\in X_0,
\end{equation}
with $f\in L^{1}_{\loc}(\mathbb{R},X)$.

The integrated semi-group is one of the major tools to investigate non-homogeneous Cauchy problems. This notion was first introduced by Ardent \cite{Arendt87a, Arendt87b}.  We refer to the books Arendt et al. \cite{Arendt01} whenever $A$ an Hille-Yosida operator. We refer to Magal and Ruan \cite{Magal-Ruan07, Magal-Ruan2018} and Thieme \cite{Thieme08} for an integrated semi-group theory whenever $A$ is not Hille-Yosida operator. We also refer to the book of Magal and Ruan \cite{Magal-Ruan2018} for more references and results on this topic.

\begin{definition} Let Assumption \ref{ASS2.2} be satisfied. Then $\left\lbrace S_A(t) \right\rbrace_{t \geq 0} \in \mathcal{L}(X)$ the \textbf{integrated semigroup generated by} $A$ is a strongly continuous family of bounded linear operators  on $X$, which is defined by
$$
S_A(t)x=(\lambda I-A_0) \int_0^t T_{A_{0}}(l) (\lambda I-A)^{-1}x dl, \forall t \geq 0.
$$
\end{definition}

In order to obtain the existence and uniqueness of solutions for \eqref{2.1} whenever $f$ is a continuous map, we will require the following assumption.

\begin{assumption}\label{ASS2.3} Assume that for any $\tau >0$ and $f\in C\left( \left[ 0,\tau %
\right] ,X\right) $ there exists $v_{f}\in C\left( \left[
0,\tau \right] ,X_{0}\right) $ an integrated (mild) solution of
\begin{equation*}
\frac{dv_{f}(t)}{dt}=Av_{f}(t)+f(t),\text{ for }t\geq 0\text{ and }v_{f}(0)=0,
\end{equation*}
that is to say that
$$
\int_0^t v_f(r)dr \in D(A), \ \forall t\geq 0
$$
and
$$
v_f(t)=A\int_0^t v_f(r)dr +\int_0^t f(r)dr , \ \forall t\geq 0.
$$
Moreover we assume that there exists a non decreasing map $\delta
:[0,+\infty) \rightarrow [0,+\infty)$ such that
\begin{equation*}
\Vert v_f(t) \Vert \leq \delta(t) \underset{s\in [0,t]}{\sup} \Vert
f(s)\Vert , \ \forall t\geq 0,
\end{equation*}
with $\delta(t)\rightarrow 0$ as $t\rightarrow 0^+$.
\end{assumption}
\begin{remark}
Note that Assumption \ref{ASS2.3} is equivalent (see \cite{Magal-Ruan09a}) to the assumption that there exists a non-decreasing map $\delta : [0,+\infty)\rightarrow [0,+\infty)$ such that for each $\tau >0$ and each $f\in C\left( \left[ 0,\tau \right] ,X\right)$ the map $t\rightarrow (S_A\ast f)(t)$ is differentiable in $[0,\tau]$ with
$$
\Vert (S_A \diamond f)(t) \Vert\leq \delta(t) \sup_{s\in [0,t]} \Vert f(s) \Vert,\ \forall t\in [0,\tau],
$$
where $(S_A\ast f)(t)$ and $(S_A \diamond f)(t)$ will be defined below in Theorem~\ref{thm2.6} and equation \eqref{2.3}.
\end{remark}
\begin{remark}\label{Remark x.4}
It is important to point out the fact Assumption \ref{ASS2.3} is also equivalent to saying that $\left\lbrace S_A(t)\right\rbrace_{t\geq 0}\subset \mathcal{L}(X,X_0)$ is of bounded semi-variation on $[0,t]$ for any $t>0$. That is to say that
$$
V^\infty(S_A,0,t):=\sup\left\lbrace \left\Vert  \sum_{i=0}^{n-1} [S_A(t_{j+1})-S_A(t_{j})]x_j \right\Vert\right\rbrace <+\infty
$$
where the supremum is taken over all partitions $0=t_0<\dots<t_n=t$ of $[0,t]$ and all elements $x_1,\dots,x_n\in X$ with $\Vert x_j \Vert\leq 1$, for $j=1,2,\ldots,n$.
Moreover the non-decreasing map $\delta : [0,+\infty)\rightarrow [0,+\infty)$ in Assumption \ref{ASS2.3} is defined by
$$
\delta(t):= \sup_{s \in [0,t]} V^\infty(S_A,0,s),\ \forall t\geq 0.
$$
\end{remark}

The following result is proved in \cite[Theorem 2.9]{Magal-Ruan09a}.
\begin{theorem}\label{thm2.6}  Let Assumptions \ref{ASS2.2} and \ref{ASS2.3} be satisfied. Then for each $\tau
>0$ and each $f\in C(\left[ 0,\tau \right] ,X)$ the map
$$
t\rightarrow \left(S_{A}\ast f\right) (t):=\int_{0}^{t}S_{A}(t-s)f(s)ds
$$
is continuously differentiable, $\left( S_{A}\ast
f\right) (t)\in D(A),\forall t\in \left[ 0,\tau \right] ,$ and if we set $%
u(t)=\frac{d}{dt}\left( S_{A}\ast f\right) (t),$ then
\begin{equation*}
u(t)=A\int_{0}^{t}u(s)ds+\int_{0}^{t}f(s)ds,\;\forall t\in \left[ 0,\tau %
\right] .
\end{equation*}
Moreover we have
\begin{equation*}
\left\| u(t)\right\| \leq \delta (t)\sup_{s\in \left[ 0,t\right] }\left\|
f(s)\right\| ,\;\forall t\in \left[ 0,\tau \right] .
\end{equation*}
Furthermore, for each $\lambda \in \left( \omega ,+\infty \right)$ we have
for each $t\in \left[ 0,\tau \right]$ that
\begin{equation} \label{2.2}
\left( \lambda I-A\right) ^{-1}\frac{d}{dt}\left( S_{A}\ast f\right)
(t)=\int_{0}^{t}T_{A_{0}}(t-s)\left( \lambda I-A\right) ^{-1}f(s)ds.
\end{equation}
\end{theorem}
From now on we will use the following notation
\begin{equation} \label{2.3}
\left( S_{A} \diamond f\right) (t):=\frac{d}{dt}\left( S_{A}\ast f\right) (t).
\end{equation}
From \eqref{2.2} and using the fact that $\left( S_{A} \diamond f\right) (t) \in X_0 $, we deduce the approximation formula
\begin{equation} \label{2.4}
\left( S_{A} \diamond f\right) (t)= \lim_{\lambda \to + \infty}\int_{0}^{t}T_{A_{0}}(t-s)\lambda \left( \lambda I-A\right) ^{-1}f(s)ds.
\end{equation}
A consequence of the approximation formula is the following
\begin{equation} \label{2.5}
\left( S_{A} \diamond f\right) (t+s)=T_{A_0}(s)\left( S_{A} \diamond f\right) (t)+ \left( S_{A} \diamond f(t+.)\right) (s), \forall t, s \geq 0.
\end{equation}

The following result is proved by Magal and Ruan \cite[Theorem 3.1]{Magal-Ruan07}, which will be constantly used and applied to the operator $A-\gamma B$ in sections 4 and 5.
\begin{theorem}[Bounded Linear Perturbation]~\\
\label{TH2.8}
Let Assumptions \ref{ASS2.2} and \ref{ASS2.3} be satisfied. Assume $L\in \mathcal{%
L}\left( X_{0},X\right) $ is a bounded linear operator.\ Then $%
A+L:D(A)\subset X\rightarrow X$ satisfies the conditions in Assumptions \ref{ASS2.2} and \ref{ASS2.3}. More precisely, if we fix $\tau _{L}>0$ such that
\begin{equation*}
\delta \left( \tau _{L}\right) \left\| L\right\| _{\mathcal{L}\left(
X_{0},X\right) }<1,
\end{equation*}
and if we denote by $\left\{ S_{A+L}(t)\right\} _{t\geq 0}$ the integrated
semigroup generated by $A+L,$ then for any $f\in C\left( \left[ 0,\tau _{L}%
\right] ,X\right)$, we have
\begin{equation*}
\left\| \frac{d}{dt}\left( S_{A+L}\ast f\right) \right\| \leq \frac{\delta
\left( t\right) }{1-\delta \left( \tau _{L}\right) \left\| L\right\| _{%
\mathcal{L}\left( X_{0},X\right) }}\sup_{s\in \left[ 0,t\right] }\left\|
f(s)\right\| ,\;\forall t\in \left[ 0,\tau _{L}\right] .
\end{equation*}
\end{theorem}
The following result is proved in \cite[Lemma 2.13]{Magal-Ruan09a}.
\begin{lemma}  Let Assumptions \ref{ASS2.2} and \ref{ASS2.3} be satisfied. Then
$$
\lim_{\lambda \to + \infty} \Vert \left( \lambda I -A \right)^{-1} \Vert_{\mathcal{L}(X)}=0.
$$
\end{lemma}
It follows that if $B \in \mathcal{L}(X_0,X)$, then for all $\lambda>0$ large enough the linear operator $\lambda I -A -B$ is invertible and its inverse can be written as follows
$$
\left(\lambda I -A -B \right)^{-1}=\left(\lambda I -A \right)^{-1}\left[ I-B\left(\lambda I -A \right)^{-1}\right]^{-1}.
$$
\section{Existence and Uniqueness of a Maximal Semiflow}
Consider now the non-autonomous semi-linear Cauchy problem
\begin{equation}\label{2.6}
U(t,s)x=x+A\int_{s}^{t}U(l,s)xdl+\int_{s}^{t}F(l,U(l,s)x)dl,\;\;\text{{}}%
t\geq s\geq 0,
\end{equation}
and the following problem
\begin{equation}
U(t,s)x=T_{A_{0}}(t-s)x+\frac{d}{dt}(S_{A}\ast F(.+s,U(.+s,s)x)(t-s),\text{ }%
t\geq s\geq 0.   \label{2.7}
\end{equation}%
We will make the following assumption.

\begin{assumption}  \label{ASS2.10} Assume that $F:\left[ 0,+\infty \right) \times \overline{D(A)}\rightarrow X$
is a continuous map such that for each $\tau _{0}>0$ and each $\xi >0,$
there exists $K(\tau _{0},\xi )>0$ such that
\begin{equation*}
\left\| F(t,x)-F(t,y)\right\| \leq K(\tau _{0},\xi )\left\| x-y\right\|
\end{equation*}
whenever $t\in \left[ 0,\tau _{0}\right] ,$\textit{\ }$y,x\in X_{0},$ and $%
\left\| x\right\| \leq \xi ,\left\| y\right\| \leq \xi .$
\end{assumption}
In the following definition $\tau $ is the blow-up time of maximal solutions
of (\ref{2.6}).
\begin{definition}[Non autonomous maximal semiflow]
\label{DE5.1} ~\\ Consider two maps $\tau :\left[ 0,+\infty \right)
\times X_{0}\rightarrow \left( 0,+\infty \right] $ and $U:D_{\tau
}\rightarrow X_{0},$ where
$$
D_{\tau }=\left\{ (t,s,x)\in \left[ 0,+\infty
\right) ^{2}\times X_{0}:s\leq t<s+\tau \left( s,x\right) \right\}.
$$
We say that $U$ is \textbf{ a maximal non-autonomous semiflow on} $X_{0}$ if $U$
satisfies the following properties
\begin{itemize}
\item[{\rm(i)}] $\tau \left( r,U(r,s)x\right) +r=\tau \left( s,x\right)
+s,\forall s\geq 0,\forall x\in X_{0},\forall r\in \left[ s,s+\tau \left(
s,x\right) \right)$.

\item[{\rm(ii)}] $U(s,s)x=x,\forall s\geq 0,\forall x\in X_{0}$.

\item[{\rm(iii)}]$U(t,r)U(r,s)x=U(t,s)x,\forall s\geq 0,\forall x\in
X_{0},\forall t,r\in \left[ s,s+\tau \left( s,x\right) \right) $ with $t\geq
r.$

\item[{\rm(iv)}] If $\tau \left( s,x\right) <+\infty ,$ then
\begin{equation*}
\lim_{t\rightarrow \left( s+\tau \left( s,x\right) \right) ^{-}}\left\|
U(t,s)x\right\| =+\infty .
\end{equation*}
\end{itemize}
\end{definition}

Set
\begin{equation*}
D=\left\{ \left( t,s,x\right) \in \left[ 0,+\infty \right) ^{2}\times
X_{0}:t\geq s\right\} .
\end{equation*}
The following theorem is the main result in this section, which was proved
in \cite[Theorem 5.2]{Magal-Ruan07}.
\begin{theorem} \label{TH2.12}
Let Assumptions \ref{ASS2.2}, \ref{ASS2.3} and \ref{ASS2.10} be satisfied. Then there exists a map $\tau
:\left[ 0,+\infty \right) \times X_{0}\rightarrow \left( 0,+\infty \right] $
and a maximal non-autonomous semiflow $U:D_{\tau }\rightarrow X_{0},$ such
that for each $x\in X_{0}$ and each $s\geq 0,$ $U(.,s)x\in C\left( \left[
s,s+\tau \left( s,x\right) \right) ,X_{0}\right) $ is a unique maximal
solution of (\ref{2.6}) (or equivalently a unique maximal solution of (\ref{2.7})). Moreover, $D_{\tau }$ is open in $D$ and the map $\left(
t,s,x\right) \rightarrow U(t,s)x$ is continuous from $D_{\tau }$ into $%
X_{0}. $
\end{theorem}


\section{Positivity and Monotonicity for a maximal semiflow}
In order to define a partial order on the Banach space $X$, we need to consider $X_+$ the positive cone of $X$. That is to say that $X_+$ is a closed convex subset of $X$ satisfying the two following properties
\begin{itemize}
\item[(i)] $\lambda x \in X_+, \; \forall \lambda \geq 0 , \forall x \in X_+$;
\item[(ii)] If $x \in X_+$ and $-x \in X_+$ then $x=0$.
\end{itemize}
Then the partial $\geq$ on $X$ is defined as follows
$$
x \geq 0 \Leftrightarrow x \in X_+.
$$
Thus, this partial order serves to compare the elements of $X$ as follows
$$
x \geq y \Leftrightarrow x-y \in X_+.
$$
The partial order on $X$ induces a partial order on $X_0$ which corresponds to the positive cone
$$
X_{0+}=X_0 \cap X_+.
$$
\begin{assumption}[Positivity] \label{ASS3.0} We assume that there exists a bounded linear operator $B \in \mathcal{L}(X_0,X)$ satisfying the two following properties

\begin{itemize}
\item[{\rm(i)}] For each $\gamma>0$ the linear operator $A-\gamma B$ is resolvent positive. That is to say that
$$
\left(\lambda I-(A- \gamma B) \right)^{-1} X_+ \subset X_+
$$
for each $\lambda >0$ large enough.
\item[\rm{(ii)}] For each $\xi >0$ and each $\sigma>0$ there exists $\gamma=\gamma(\xi,\sigma)>0$, such that
$$
 x \geq 0 \Rightarrow F(t,x)+\gamma Bx \geq 0
$$
whenever $x \in X_0$,  $ \Vert x \Vert \leq \xi $ and $t \in [0, \sigma]$.
\end{itemize}
\end{assumption}

\begin{remark} \label{REM3.3} For densely defined Cauchy problems, the standard conditions related to the positivity of solutions can be obtained if we take $B=I$, the identity, in Assumption \ref{ASS3.0}. We also note that the positivity for a delay differential equation was discussed in \cite[Example 3.6]{Magal-Ruan09a}, where $B$ is not equal to the identity.
\end{remark}

Assumption \ref{ASS3.0} ${\rm(i)}$ guarantees that the semigroup $\left\{ T_{\left( A- \gamma B \right)_{0}}(t)\right\} _{t\geq 0}$ is a $C_0$-semigroup of positive operators on $X_{0}$. By using the classical semigroup approximation formula
$$
 T_{\left( A- \gamma B \right)_{0}}(t)x=\lim_{n \to \infty} \left( \dfrac{n}{t} \right)^n
 \left( \dfrac{n}{t} -(A- \gamma B) \right)^{-n}x \geq 0, \forall t >0, \forall x \in X_{ 0+}.
$$
As a consequence of formula \eqref{2.4} we deduce that for each $t \in [0, \tau]$,
$$
\left( S_{(A-\gamma B)} \diamond f \right)(t)=\lim_{\lambda \to \infty} \int_0^t T_{\left( A- \gamma B \right)_{0}}(t-s) \lambda \left(\lambda I-(A- \gamma B) \right)^{-1} f(s)ds \geq 0,
$$
whenever $f \in C\left([0, \tau],X_+ \right)$. As a consequence
\begin{equation}\label{eq3.1}
f \leq g \Rightarrow \left( S_{(A-\gamma B)} \diamond f \right) \leq \left( S_{(A-\gamma B)} \diamond g \right),
\end{equation}
whenever $g-f \in C\left([0, \tau],X_+ \right)$ and $f \in C\left([0, \tau],X \right)$.

The following result has been studied in \cite[Proposition 3.5]{Magal-Ruan09a}.
\begin{theorem}[Positive semiflow]\label{TH3.1}~\\
Let Assumptions \ref{ASS2.2}, \ref{ASS2.3}, \ref{ASS2.10} and \ref{ASS3.0} be satisfied. Then the maximal semiflow generated by (\ref{2.6}) is non negative. That is to say that for each $x \in X_0$ and $s \geq 0$
$$
x \geq 0  \Rightarrow  U(t,s)x \geq 0,
$$
whenever $t \in \left[s,  s+\tau \left( s,x\right) \right)$.
\end{theorem}

In order to derive a monotone semiflow property we need to extend the Assumption \ref{ASS3.0} as follows.

\begin{assumption}[Monotonicity] \label{ASS3.1} We assume that there exists a bounded linear operator $B \in \mathcal{L}(X_0,X)$ satisfying the two following properties

\begin{itemize}
\item[{\rm(i)}] For each $\gamma>0$ the linear operator $A-\gamma B$ is resolvent positive. That is to say that
$$
\left(\lambda I-(A- \gamma B) \right)^{-1} X_+ \subset X_+
$$
for each $\lambda >0$ large enough.
\item[{\rm(ii)}] For each $\xi >0$ and each $\sigma>0$ there exists $\gamma=\gamma(\xi,\sigma)>0$, such that
$$
0 \leq x \leq y \Rightarrow 0 \leq F(t,x)+\gamma Bx \leq F(t,y)+\gamma By
$$
whenever $x,y \in X_0$,  $ \Vert x \Vert \leq \xi $, $ \Vert y \Vert \leq \xi $ and $t \in [0, \sigma]$.
\end{itemize}
\end{assumption}

 Assumption \ref{ASS3.1} ${\rm(ii)}$ definitely means that the map $x \to  F(t,x)+\gamma Bx$ is non negative and monotone increasing on $B_{X_0}(0,\xi) \cap X_+$ for each $t \in [0, \sigma]$. Consequently we obtain the following result.
\begin{theorem}[Monotone Semiflow]\label{TH3.4}~\\
Let Assumptions \ref{ASS2.2}, \ref{ASS2.3}, \ref{ASS2.10} and \ref{ASS3.1} be satisfied. Then the maximal semiflow generated by (\ref{2.6}) is non negative and monotone increasing. That is to say that for each $x, y \in X_0$ and $s \geq 0$
$$
0 \leq x \leq y \Rightarrow 0 \leq U(t,s)x \leq U(t,s)y,
$$
whenever $t \in \left[s, \min\left( s+\tau \left( s,x\right),s+\tau \left( s,y \right) \right) \right)$
\end{theorem}
\begin{proof}
Without loss of generality we can assume that $s=0$ and $x\in X_{0+}.$
Moreover, using the semiflow property, it is sufficient to prove that there
exists $\sigma_0 \in \left( 0,\tau \left( 0,x\right) \right) $ such that $%
U(t,0)x\in X_{0+},$ $\forall t\in \left[ 0,\sigma_0 \right] .$ Let $x, y\in
X_{0+}$ be fixed. Set
$$
\xi :=2 \left( \left\Vert x \right\Vert+ \Vert y \Vert +1 \right) .
$$
Let $\gamma >0$ such that
\begin{equation*}
u \leq  v   \Rightarrow F(t,u)+\gamma Bu \leq F(t,v)+\gamma Bv
\end{equation*}%
when $u,v \in X_{0+}$, $\left\Vert u \right\Vert \leq \xi $, $\left\Vert v \right\Vert \leq \xi $ and $t\in \lbrack 0,1]$. \\
Let us fix $\tau _{\gamma }>0$ such that
$$
\gamma \Vert B \Vert_{\mathcal{L}(X_0,X)} \delta \left(\tau _{\gamma } \right)<1.
$$ For each $\sigma \in \left( 0,\tau _{\gamma
}\right) ,$ define
\begin{equation*}
E^{\sigma }=\left\{ \varphi \in C\left( \left[ 0,\sigma \right]
,X_{0+}\right) :\left\Vert \varphi (t)\right\Vert \leq \xi ,\;\;\forall t\in %
\left[ 0,\sigma \right] \right\} .
\end{equation*}%
Then it is sufficient to consider the fixed point problem
\begin{equation*}
u(t)=T_{(A-\gamma B)_{0}}(t)x + S_{(A-\gamma B)}\diamond ( F(.,u(.))+\gamma
Bu(.)) (t)=:\Psi_x (u)(t), \forall t\in [0,\sigma].
\end{equation*}
Moreover, by using Theorem \ref{TH2.8}, for each $\varphi \in E^{\sigma }$
and each $t\in \left[ 0,\sigma \right] $, we deduce that
\begin{eqnarray*}
\left\Vert \Psi_x (\varphi )(t)\right\Vert &=&\left\Vert T_{(A-\gamma
B)_{0}}(t)x+ S_{(A-\gamma B)}\diamond \left(F(.,\varphi (.))+\gamma B \varphi
(.)\right) (t)\right\Vert \\
&\leq &\left\Vert T_{(A-\gamma B)_{0}}(t)x\right\Vert \\
&& +\frac{\delta (t)}{1-\gamma \Vert B \Vert_{\mathcal{L}(X_0,X)}  \delta (\tau _{\gamma })}\sup_{s\in [
0,t] }\left\Vert F(s,\varphi(s))+\gamma B \varphi (s)\right\Vert \\
&\leq &\sup_{t\in \left[ 0,\sigma \right] }\left\Vert T_{(A-\gamma
B)_{0}}(t)x\right\Vert \\
&&+\frac{\delta (\sigma )}{1-\gamma \Vert B \Vert_{\mathcal{L}(X_0,X)}  \delta (\tau
_{\gamma })}\left[ \sup_{s\in \left[ 0,\sigma \right] }\left\Vert
F(s,0)\right\Vert +\left[ K(1,\xi )+\gamma \Vert B \Vert_{\mathcal{L}(X_0,X)} \right] \xi \right] .
\end{eqnarray*}%
Hence, there exists $\sigma _{1}\in \left( 0,1\right) $ such that
\begin{equation*}
\Psi_x(E^{\sigma })\subset E^{\sigma } \text{ and } \Psi_y(E^{\sigma })\subset E^{\sigma },\;\forall \sigma \in \left( 0,\sigma
_{1}\right] .
\end{equation*}%
Therefore, for each $\sigma \in \left( 0,\sigma _{1}\right] $ and each pair $%
\varphi ,\psi \in E^{\sigma },$ we have for $t\in \left[ 0,\sigma \right] $ and $z=x,y$
that
\begin{eqnarray*}
\left\Vert \Psi_z (\varphi )(t)-\Psi_z (\psi )(t)\right\Vert  =\left\Vert
\left( S_{(A-\gamma B)}\diamond \left[ F(.,\varphi (.))-F(.,\psi (.))+\gamma B
\left( \varphi -\psi \right) (.)\right] \right) (t)\right\Vert
\end{eqnarray*}%
and by using Assumption \ref{ASS2.10} we obtain
\begin{eqnarray*}
&& \left\Vert \Psi_z (\varphi )(t)-\Psi_z (\psi )(t)\right\Vert \\
&& \quad\quad \leq \frac{\delta (t)}{1-\gamma \Vert B \Vert_{\mathcal{L}(X_0,X)}
\delta (\tau_{\gamma })}\left[ K(1,\xi )+\gamma \Vert B \Vert_{\mathcal{L}(X_0,X)} \right] \sup_{s\in \left[ 0,\sigma
\right] }\left\Vert \left( \varphi -\psi \right) (s)\right\Vert .
\end{eqnarray*}
Thus, there exists $\sigma _{2}\in \left( 0,\sigma _{1}\right] $ such that $%
\Psi_z (E^{\sigma _{2}})\subset E^{\sigma _{2}}$ and $\Psi_z $ is a contraction
strict on $E^{\sigma _{2}}$.

Now choosing the constant functions $\varphi  \equiv x$ and $\psi \equiv y $ we obtain
$$
\Psi_x(\varphi) \leq \Psi_y (\varphi) \leq  \Psi_y(\psi).
$$
Since $\Psi$ is monotone, it follows that
$$
\Psi_x^2(\varphi) \leq \Psi_y ( \Psi_x(\varphi) ) \leq \Psi_y ^2(\psi).
$$
Now by induction, we obtain
$$
\Psi_x^n(\varphi) \leq  \Psi_y ^n(\psi), \forall n \in \mathbb{N} \setminus \left\lbrace 0 \right\rbrace.
$$
By taking the limit when $n$ goes to $+ \infty$ we obtain
$$
U(t,0)x=\lim_{n \to + \infty} \Psi_x^n(\varphi)(t) \leq \lim_{n \to + \infty} \Psi_y^n(\varphi)(t)=U(t,0)y.
$$
\end{proof}

Recall that the positive cone $X_{+}$ is said to be \textbf{normal} if there exists a norm $\left\Vert .\right\Vert _{1}$
equivalent to $\left\Vert .\right\Vert $, which is monotone. That is to say that
\begin{equation*}
\forall x,y\in X,\text{ }0\leq x\leq y\Rightarrow \left\Vert x\right\Vert
_{1}\leq \left\Vert y\right\Vert _{1}.
\end{equation*}

\begin{lemma}
Let Assumptions \ref{ASS2.2}, \ref{ASS2.3}, \ref{ASS2.10} and \ref{ASS3.1} be satisfied. Assume in addition that $X_+$ is normal. Then for each $x, y \in X_0$ and $s \geq 0$  the maximal time of existence of the semiflow $U$ satisfies the following properties
$$
0 \leq x \leq y \Rightarrow \tau(s,x) \geq \tau(s,y).
$$
\end{lemma}
\begin{proof}
By Theorem~\ref{TH3.4}, we see that
$$
 0 \leq U(t,s)x \leq U(t,s)y,\ \forall t \in \left[s,s+ \min\left( \tau \left( s,x\right),\tau \left( s,y \right) \right) \right).
$$
Since the equivalent norm $\left\Vert .\right\Vert _{1}$ is monotone, it follows that
$$\left\Vert U(t,s)x\right\Vert_{1}\leq \left\Vert U(t,s)y\right\Vert _{1},\ \forall t \in \left[s,s+ \min\left( \tau \left( s,x\right),\tau \left( s,y \right) \right) \right).$$
Assume by contradiction that $\tau(s,x) < \tau(s,y)$. Then we have
$$
+\infty=\lim_{t\rightarrow \left( s+\tau \left( s,x\right) \right) ^{-}}\left\|
U(t,s)x\right\|_{1} \leq \lim_{t\rightarrow \left( s+\tau \left( s,x\right) \right) ^{-}}\left\|
U(t,s)y\right\|_{1} <+\infty,
$$
which is a contradiction.
\end{proof}

\section{Comparison principle}

\begin{proposition}[Lower Solution: Integral form]\label{pro4.1}
Let Assumptions \ref{ASS2.2}, \ref{ASS2.3}, \ref{ASS2.10} and \ref{ASS3.1} be satisfied, and $v \in C([s,s+\hat{\tau}], X_{0+})$ with $s \geq 0$ and $\hat{\tau} \geq 0$. Assume that for each $\gamma>0$ large enough and each $t \in [s,s+\hat{\tau}]$,
$$
v(t) \leq T_{(A-\gamma B)_0}(t-s)x+\left(S_{(A-\gamma B)} \diamond (\gamma B+F)(s+.,v(s+.)) \right)(t-s).
$$
Then
$$
v(t) \leq U(t,s)x,
$$
whenever $t \in [s,\min(s+\hat{\tau}, s+\tau(s,x)))$.
\end{proposition}
\begin{proof}Our arguments will be similar to those in Theorem~\ref{TH3.4}.
Without loss of generality we can assume that $s=0$ and $x\in X_{0+}.$
Moreover, using the semiflow property, it is sufficient to prove that there
exists  $\sigma_0 \in \left( 0,\min(\hat{\tau}, \tau \left( 0,x\right)) \right)$ such that
$$
v(t) \leq U(t,0)x, \forall \  t\in \left[ 0,\sigma_0 \right].
$$
Let $x\in X_{0+}$ be fixed; $\tau _{\gamma}$ and $\gamma$ be given in the proof of Theorem~\ref{TH3.4}. For each $\sigma \in \left( 0,\tau _{\gamma
}\right)$, we define
\begin{equation*}
\Psi_x (v)(t):=T_{(A-\gamma B)_{0}}(t)x +  S_{(A-\gamma B)}\diamond ( F(.,v(.))+\gamma
Bv(.)) (t), \ \forall \ t\in [0,\sigma].
\end{equation*}
Then for all $t\in [0,\sigma]$, we have
$$v(t)\leq \Psi_x (v), $$
$$\Psi_x (v)\leq \Psi_x(\Psi_x (v))=\Psi_x^2 (v),$$
and
$$\Psi_x^{(n-1)} (v)\leq \Psi_x(\Psi_x^{(n-1)} (v))=\Psi_x^n (v),\ \forall\ n\in N.$$
That is to say that for all $t\in [0,\sigma]$, it follows that
$$v(t)\leq \Psi_x (v)\leq \Psi_x^2 (v) \leq \Psi_x^3 (v)\leq \ldots \leq \Psi_x^n (v)\rightarrow U(t,0)x\ \mbox{as}\ n\rightarrow \infty. $$
The proof is complete.
\end{proof}
\begin{proposition}[Upper Solution: Integral form]
Let Assumptions \ref{ASS2.2}, \ref{ASS2.3}, \ref{ASS2.10} and \ref{ASS3.1} be satisfied, and $v \in C([s,s+\hat{\tau}], X_{0+})$ with $s \geq 0$ and $\hat{\tau} \geq 0$. Assume that for each $\gamma>0$ large enough and each $t \in [s,s+\hat{\tau}]$,
$$
v(t) \geq T_{(A-\gamma B)_0}(t-s)x+\left(S_{(A-\gamma B)} \diamond (\gamma B+F)(s+.,v(s+.)) \right)(t-s).
$$
Then
$$
\tau(s,x) >\hat{\tau}
$$
and
$$
v(t) \geq U(t,s)x, \forall t \in [s,s+\hat{\tau}].
$$
\end{proposition}
Let $I$ be an interval in $\mathbb{R}$. We recall that $v \in C(I, D(A) )$ if and only if
$$
v(t) \in D(A), \forall t \in I,
$$
and the map $t \to Av(t)$ is continuous from $I$ in $X$.
\begin{proposition}[Lower Solution: Differential form] \label{PROP4.3}
Let Assumptions \ref{ASS2.2}, \ref{ASS2.3}, \ref{ASS2.10} and \ref{ASS3.1} be satisfied.
Assume that $s \geq 0$, $\hat{\tau} \geq 0$, and $v \in C^1([s,s+\hat{\tau}], X) \cap C([s,s+\hat{\tau}], D(A) )$ is a non-negative function that is
$$
v(t) \geq 0, \forall t \in [s,s+\hat{\tau}].
$$
If
$$
v'(t) \leq Av(t)+ F(t,v(t)), \forall t \in [s,s+\hat{\tau}], \text{ and } v(s)=x \in X_{0+},
$$
then
$$
v(t) \leq U(t,s)x, \forall t \in [s,\min(s+\hat{\tau}, s+\tau(s,x))).
$$
\end{proposition}
\begin{proof}
Let
$$
h(t):=v'(t) -[ Av(t)+ F(t,v(t))], \forall t \in [s,s+\hat{\tau}].
$$
Then $$h(.)\in C([s,s+\hat{\tau}], X)\ \mbox{and}\ h(.)\leq 0 \ \mbox{on}\ [s,s+\hat{\tau}],$$
and
$$
v'(t) = Av(t)+ F(t,v(t))+h(t), \forall t \in [s,s+\hat{\tau}], \text{ and } v(s)=x \in X_{0+}.
$$
Let $\gamma$ be given in the proof of Theorem~\ref{TH3.4}. Then for all $t \in [s,s+\hat{\tau}]$, we have
\begin{eqnarray*}
v(t)&=&T_{(A-\gamma B)_0}(t-s)x+\left(S_{(A-\gamma B)} \diamond (\gamma B v(s+.)+F(s+.,v(s+.)+h(s+.) \right)(t-s),\nonumber\\
&\leq& T_{(A-\gamma B)_0}(t-s)x+\left(S_{(A-\gamma B)} \diamond (\gamma B v(s+.)+F(s+.,v(s+.) \right)(t-s),\nonumber
\end{eqnarray*}
where we have used the fact \eqref{eq3.1}. By using Proposition \ref{pro4.1}, we complete the proof.

\end{proof}

\begin{proposition}[Upper Solution: Differential form] \label{PROP4.4}
Let Assumptions \ref{ASS2.2}, \ref{ASS2.3}, \ref{ASS2.10} and \ref{ASS3.1} be satisfied. Assume that $s \geq 0$, $\hat{\tau} \geq 0$, and $v \in C^1([s,s+\hat{\tau}], X) \cap C([s,s+\hat{\tau}], D(A) )$ is a non-negative function that is
$$
v(t) \geq 0, \forall t \in [s,s+\hat{\tau}].
$$
If
$$
v'(t) \geq A v(t)+ F(t,v(t)), \forall t \in [s,s+\hat{\tau}], \text{ and } v(s)=x \in X_{0+},
$$
then
$$
\tau(s,x) >\hat{\tau}
$$
and
$$
v(t) \geq U(t,s)x, \forall t \in [s,s+\hat{\tau}].
$$
\end{proposition}
\begin{theorem}[Increasing and Decreasing Solutions] \label{TH5.1} Let Assumptions \ref{ASS2.2}, \ref{ASS2.3}, \ref{ASS2.10} and \ref{ASS3.1} be satisfied.  Assume that $F$ is independent of time $t$, and $U:D_U \subset \left[0,+ \infty \right) \times X_0 \to X_0 $ is the autonomous maximal semiflow generated by the abstract Cauchy problem
$$
u^\prime(t)=Au(t)+F(u(t)), \text{ for }t \geq 0, u(0)=x \in \overline{D(A)}.
$$
Assume in addition that
$$
x \in D(A) \cap X_{+}.
$$
Then we have the following properties
\begin{itemize}
\item[{\rm(i)}]$Ax+F(x) \geq 0, \forall t \geq 0$ implies that the map $t \to U(t)x$ is increasing.
\item[{\rm(ii)}]$Ax+F(x) \leq 0, \forall t \geq 0$ implies that the map $t \to U(t)x$ is decreasing.
\end{itemize}
\end{theorem}
\begin{proof} Set
$$
v(t)=x, \forall t \geq 0.
$$
Observe that
$$
v'(t)=0 \leq Ax+F(x)=Av(t)+F(v(t)), \forall t \geq 0.
$$
Now by using the differential form of comparison principle in Proposition \ref{PROP4.3}, we obtain
$$
x=v(t) \leq U(t)x, \forall t \geq 0
$$
and by applying $U(t')$ on both side of this last inequality we obtain
$$
U(t')x \leq U(t'+t) x, \forall t ,t' \geq 0,
$$
and the result follows.
\end{proof}

\section{Applications to age structured models}
\subsection{SIR epidemic model with infection age}
In this subsection, we intend to use the comparison principle to determine a positively invariant sub-region for the SIR epidemic model with age of infection
\begin{equation} \label{6.3}
\left\lbrace
\begin{array}{l}
S'(t)=\gamma-\nu_S S(t)-\eta S(t)\int_0^{+\infty} \beta(a)i(t,a)da,\\
\partial_t i(t,a)+\partial_a i(t,a)=-\nu_I(a)i(t,a), \text{{\rm for }} a \in (0,\infty)\\
i(t,0)=\eta S(t)\int_0^{+\infty} \beta(a)i(t,a)da, \\
S(0)=S_0\geq0,\ i(0,.)=i_0 \in \Li^1_+(\left(0,+\infty \right), \mathbb{R}).
\end{array}
\right.
\end{equation}
This model was first proposed by  Kermack and McKendrick  \cite{Kermack-McKendrick} in 1927. {The} global dynamic was studied by Magal, McCluskey and Webb \cite{MMW10}. The existence of {an invariant sub-region turn to be crucial to the investigation of the uniform persistence}. Some results based on comparison principle were {stated} in \cite{MMW10} without proof. In the following, we will explain how to derive these invariant sub {regions} by using the above comparison principle.

\begin{assumption}\label{ASS6.16} We assume that the following conditions are satisfied
\begin{itemize}
\item[{\rm(i)}]$\gamma>0$, $\nu_S>0$, $\eta>0$;
\item[{\rm(ii)}]The function $a\rightarrow \beta(a)$ is bounded and uniformly
continuous from $[0,\infty)$ to $[0,\infty)$;
\item[{\rm(iii)}]The function $ \nu_I(\cdot)\in \Li^{+\infty }_+\left(0,+\infty \right)$, and $\nu_I(a)\geq \delta$ for almost every $a\geq 0$ for some $\delta>0$.
\end{itemize}
\end{assumption}

\noindent \textbf{Volterra formulation:} By integrating along the characteristic the $i$-equation we deduce that
\begin{equation*}
i(t,a)=
\left\lbrace
\begin{array}{l}
e^{-\int_{a-t}^{a} \nu_I(\sigma)d\sigma} i_0(a-t), \text{{\rm if }} a-t \geq 0,\\
e^{-\int_{0}^{a} \nu_I(\sigma)d\sigma} B(t-a), \text{{\rm if }} t-a \geq 0,
\end{array}
\right.
\end{equation*}
and $t \to B(t)$ is the unique continuous function satisfying the following Volterra integral equation  for each $t\geq 0$
$$
B(t)=S(t) \left \lbrace \int_t^\infty \beta(a)  e^{-\int_{a-t}^{a} \nu_I(\sigma)d\sigma} i_0(a-t) da + \int_0^t  \beta(a)  e^{-\int_{0}^{a} \nu_I(\sigma)d\sigma} B(t-a) da \right \rbrace.
$$

\noindent \textbf{Integrated semigroup formulation:}
Treating $S(t)$ as a known function for the $i$-equation, we are going to rewrite the $i$-part of system \eqref{6.3} as an Abstract Cauchy problem. Set
$$
{ X=\mathbb{R} \times \Li^1(\left(0,+\infty \right),\mathbb{R})}
$$
endowed with the usual product norm. Consider the linear operator $A:D(A)\subset X \to X$
$$
A
\left(
\begin{array}{c}
0_{\mathbb{R}}\\
i
\end{array}
\right)
=
\left(
\begin{array}{c}
-i(0)\\
-i^\prime -\nu_I i
\end{array}
\right)
$$
with
$$
D(A)=\left\{ 0_{\mathbb{R}} \right\} \times \W^{1,1}(\left(0, +\infty \right), \mathbb{R}).
$$
Then the closure of the domain of $A$ is
$$
X_0:=\overline{D(A)}=\left\{ 0_{\mathbb{R}} \right\} \times \Li^{1}(\left(0, +\infty \right), \mathbb{R}).
$$
{Define $F:\mathbb{R} \times X_0 \to X$ by}
$$
{ F
\left(S,\left(
\begin{array}{c}
0_{\mathbb{R}}\\
i
\end{array}
\right)\right)}
=
\left(
\begin{array}{c}
\eta S\int_0^{+\infty} \beta(a)i(a)da\\
0_{L^1}
\end{array}
\right).
$$
By identifying $i(t,.)$ with $v(t):=
\left(
\begin{array}{c}
0_{\mathbb{R}}\\
i(t,.)
\end{array}
\right)$ we can rewrite the $i$-equation in \eqref{6.3} as the following abstract Cauchy problem
\begin{equation*}
v^\prime(t)=Av(t)+F(S(t),v(t)), \text{ for } t \geq 0, v(0)=\left(
\begin{array}{c}
0_{\mathbb{R}}\\
i_0
\end{array}
\right) \in X_0.
\end{equation*}

\begin{lemma}\label{PROP6.17}Assume that the initial value $S(0)=S_0>0$. Then there exists $S_{+}>S_{-}>0$ such that
$$S_{-} \leq S(t)\leq S_{+},\ \mbox{for}\ t\geq 0.$$
\end{lemma}
\begin{remark} The fact that we assume that $S(0)=S_0>0$ is not a restriction. Indeed, if $S(0)=S_0=0$, one can prove that
$$
S(t)>0, \forall t >0.
$$
Therefore we always replace $0$ by some small positive time $t>0$ and assume that $S_0>0$.
\end{remark}
\begin{proof}We first establish the upper bound of $S(t)$. It is easy to see that
$$S'(t)\leq \gamma-\nu_S S(t).$$
Thus,
\begin{equation*} \label{6.4}
S(t) \leq e^{-\nu_S t}S_0+\int_0^t\gamma e^{-\nu_S (t-\theta)}d\theta= e^{-\nu_S t}S_0+\frac{\gamma}{\nu_S}[1-e^{-\nu_S t}],\ t\geq 0,
\end{equation*}
which implies that
\begin{equation*} \label{6.5}
S(t) \leq S_0+\frac{\gamma}{\nu_S}:=S_{+},\ t\geq 0.
\end{equation*}
Next, we establish the lower bound of $S(t)$.
Define
$$
I(t)=\int_0^{+\infty}i(t,a)da.
$$
By using classical solutions of \eqref{6.3} we deduce that
$$
 \frac{d}{d t}\int_0^{\infty}i(t,a) da= \int_0^{\infty}[-\partial_a i(t,a)-\nu_I(a)i(t,a)] da.
$$
{Using} integration by parts, we deduce that
\begin{equation*} \label{6.6}
I'(t)=\eta S(t)\int_0^{+\infty} \beta(a)i(t,a)da-\int_0^{+\infty} \nu_I(a)i(t,a)da.
\end{equation*}
The formula is true for any mild solution by continuity of the semi-flow generated by \eqref{6.3}, and by density of initial distribution giving a classical solution.

From the above formula we deduce that
\begin{equation*} \label{6.7}
(S+I)'(t)\leq \gamma-\nu_S S(t)-\delta I(t).
\end{equation*}
This implies that $t\rightarrow S(t)+I(t)$ is bounded above by a constant $M>0$.
On the other hand, we have
\begin{eqnarray*}
S'(t) &\geq& \gamma-\nu_S S(t)-\eta \parallel\beta\parallel_{\infty}S(t)I(t)\\
&\geq & \gamma-[\nu_S +\eta \parallel\beta\parallel_{\infty}M]S(t).
\end{eqnarray*}
This implies that
$$\liminf_{t\rightarrow \infty}S(t)\geq \frac{\gamma}{\nu_S +\eta \parallel\beta\parallel_{\infty}M}.$$
\end{proof}

By Lemma~\ref{PROP6.17} and the comparison principle, we have the following result:

\begin{theorem}\label{cor6.18}Assume that $i_{\pm}(t,a)$ satisfies
\begin{equation*} \label{6.8}
\left\lbrace
\begin{array}{l}
\partial_t i_{\pm}(t,a)+\partial_a i_{\pm}(t,a)=-\nu_I(a)i_{\pm}(t,a), \text{ for } a \geq 0,\ t\geq0,\\
i_{\pm}(t,0)=\eta S_{\pm}\int_0^{+\infty} \beta(a)i_{\pm}(t,a)da, \\
i_{\pm}(0,.)=i_0 \in \Li^1_+(\left(0,+\infty \right), \mathbb{R}).
\end{array}
\right.
\end{equation*}
Then
\begin{equation} \label{6.9}
i_{-}(t,a)\leq i(t,a) \leq i_{+}(t,a),\ \mbox{for}\ t\geq0, \ \mbox{for a. e.}\ a\geq 0.
\end{equation}
\end{theorem}
\begin{remark} The above Theorem is general and can be used for example in the context of ecology whenever $i(t,a)$ is the distribution of population and $S(t)$ 
represents a time dependent reproduction rate.
\end{remark}
\begin{remark} \label{REM6.19} One can prove the above result by using a Volterra integral formulation approach. Consider the function $t \to B(t)$ is the unique continuous function satisfying for each $t\geq 0$
$$
B(t)=S(t) \left \lbrace \int_t^\infty \beta(a)  e^{-\int_{a-t}^{a} \nu_I(\sigma)d\sigma} i_0(a-t) da + \int_0^t   \beta(a) e^{-\int_{0}^{a} \nu_I(\sigma)d\sigma} B(t-a) da \right \rbrace.
$$
We can for example derive the following upper estimation
$$
B(t) \leq S^+ \left \lbrace \int_t^\infty \beta(a)  e^{-\int_{a-t}^{a} \nu_I(\sigma)d\sigma} i_0(a-t) da + \int_0^t  e^{-\int_{0}^{a} \nu_I(\sigma)d\sigma} B(t-a) da \right \rbrace.
$$
{By} using a standard iteration procedure
$$
B^n(t)=S^+ \left \lbrace \int_t^\infty \beta(a)  e^{-\int_{a-t}^{a} \nu_I(\sigma)d\sigma} i_0(a-t) da + \int_0^t  e^{-\int_{0}^{a} \nu_I(\sigma)d\sigma} B^{n-1}(t-a) da \right \rbrace,
$$
with
$$
B^0(t)=B(t)
$$
it follows that the sequence $B^n$ is increasing, that is to say
$$
B^n(t)\leq B^{n+1}(t), \forall t \geq 0, \forall n \geq 0.
$$
{Letting $n$ go to $\infty$, the} sequence $B^n$ converges and we have
$$
\lim_{n \to \infty }B^n(t) =B^+(t) \text{ locally uniformly in } t \in \mathbb{R},
$$
where $t \to B^+(t)$ is the unique continuous function satisfying the Volterra integral equation
$$
B^+(t)= S^+ \left \lbrace \int_t^\infty \beta(a)  e^{-\int_{a-t}^{a} \nu_I(\sigma)d\sigma} i_0(a-t) da + \int_0^t  e^{-\int_{0}^{a} \nu_I(\sigma)d\sigma} B^+(t-a) da \right \rbrace.
$$
Now we deduce that
$$
i(t,a) \leq i^+(t,a),
$$
{where}
\begin{equation}
 i^+(t,a)=
\left\lbrace
\begin{array}{l}
e^{-\int_{a-t}^{a} \nu_I(\sigma)d\sigma} i_0(a-t), \text{{\rm if }} a-t \geq 0,\\
e^{-\int_{0}^{a} \nu_I(\sigma)d\sigma} B^+(t-a), \text{{\rm if }} t-a \geq 0.
\end{array}
\right.
\end{equation}

\end{remark}

As a consequence we can determine the invariant sub-region that can serve to describe the uniform persistence properties of the system. Let
$$
\Gamma_I^{\pm}(a)=\eta S_{\pm}\int_{a}^{\infty}e^{-\int_a^\theta[\nu_I(\sigma)+\lambda^{\pm}]d \sigma}\beta(\theta)d\theta,
$$ where $\lambda^{\pm}$ is chosen to satisfy that $\Gamma_I^{\pm}(0)=1$. That is to say that $\lambda^{\pm} \in \mathbb{R}$ satisfies
$$
\eta S_{\pm}\int_{0}^{\infty}e^{-\int_0^\theta[\nu_I(\sigma)+\lambda^{\pm}]d \sigma}\beta(\theta)d\theta=1.
$$
\begin{remark} $\lambda^{\pm} \in \mathbb{R}$ is the dominant eigenvalue of  $v \to  Av+F(S_{\pm},v)$. The above integral equation corresponds to the characteristic equation of this linear operator.
\end{remark}
\begin{lemma} For each $t \geq 0$,
\begin{equation} \label{6.10}
\int_0^{\infty}\Gamma_I^{\pm}(a)i_{\pm}(t,a) da=e^{\lambda^{\pm}t}\int_0^{\infty}\Gamma_I^{\pm}(a)i_0(a) da.
\end{equation}
\end{lemma}
\begin{proof}
The function $ a \to \Gamma_I^{\pm}(a)$ satisfies
\begin{equation*}
\left\lbrace
\begin{array}{ll}
(\Gamma_I^{\pm})'(a)=[\nu_I(a)+\lambda^{\pm}]\Gamma_I^{\pm}(a)-\eta S_{\pm}\beta(a),\ \mbox{for a. e.}\ a\geq 0,\\
\Gamma_I^{\pm}(0)=1.
\end{array}
\right.
\end{equation*}
By using classical solutions of \eqref{6.8} we deduce that
$$
 \frac{d}{d t}\int_0^{\infty}\Gamma_I^{\pm}(a)i_{\pm}(t,a) da= \int_0^{\infty}\Gamma_I^{\pm}(a)[-\partial_a i_{\pm}(t,a)-\nu_I(a)i_{\pm}(t,a)] da.
$$
{Using} integration by parts, it follows that
\begin{eqnarray*}
 \frac{d}{d t}\int_0^{\infty}\Gamma_I^{\pm}(a)i_{\pm}(t,a) da
&=&\Gamma_I^{\pm}(0) i_{\pm}(t,0)+\int_0^{\infty}(\Gamma_I^{\pm})'(a)i_{\pm}(t,a) da\\
&-&\int_0^{\infty}\Gamma_I^{\pm}(a)\nu_I(a)i_{\pm}(t,a) da.
\end{eqnarray*}
Using the facts $\Gamma_I^{\pm}(0)=1$ and $i_{\pm}(t,0)=\eta S_{\pm}\int_0^{+\infty} \beta(a)i_{\pm}(t,a)da$, we have
\begin{equation*}
\frac{d}{d t}\int_0^{\infty}\Gamma_I^{\pm}(a)i_{\pm}(t,a) da=\lambda^{\pm}\int_0^{\infty}\Gamma_I^{\pm}(a)i_{\pm}(t,a) da,
\end{equation*}
{and} the result follows {from} the fact that the set of initial values giving a classical solution is dense in $L^1$.
\end{proof}

Assume that $\beta \neq 0$. Let
$$
a^\star:=\sup \left\{a>0: \int_a^{\infty}\beta(\sigma) e^{-\sigma} d\sigma >0  \right\} \in (0, \infty].
$$
\begin{remark} In practice, the number $a^\star$ corresponds to the maximal value {at which some new infection can still be produced} (possibly in the future). { This  means} that $a^\star=\infty$ if and only if for each $a \geq 0$ there exists $\widehat{a}>a$ such that
$$
\beta(\widehat{a})>0.
$$
If $a^\star<\infty$ then
$$
\beta(a)=0, \forall a \geq a^\star,
$$
and for each $a \in [0,a^\star)$ there exists $\widehat{a} \in (a, a^\star)$ such that
$$
\beta(\widehat{a})>0.
$$
\end{remark}
Define the interior sub-domain
$$
\widehat{M}_0=\left \{i \in L^1_+(0,+\infty): \int_0^{a^\star} i(a)da>0 \right\},
$$
and the boundary sub-domain
$$
\partial \widehat{M}_0=\left \{i \in L^1_+(0,+\infty): \int_0^{a^\star} i(a)da=0 \right\}.
$$
Actually the boundary domain $\partial \widehat{M}_0$ corresponds to a case where the distribution $i_0$ contains only people that will not produce new infected individuals. {However, some infected people in the interior region} will produce new infected {individuals}. Due to the irreducible structured of the semiflow generated by the $i$-equation  {we can obtain the invariance of $\partial \widehat{M}_0$ and $\widehat{M}_0$.}

The following theorem was stated without proof in Magal et al. \cite[Lemma 2.3]{MMW10}. The {comparison principle was actually applied implicitly to deriving} such a result.
\begin{theorem}\label{Th6.23} The domains $\left[ 0,\infty \right) \times \widehat{M}_0 $ and  $\left[ 0,\infty \right) \times  \partial \widehat{M}_0 $  are positively invariant by the semiflow generated by \eqref{6.3}. {That is} to say that
$$
\int_0^{a^\star} i_0(a)da>0 \Rightarrow \int_0^{a^\star} i(t,a)da>0, \forall t \geq 0
$$
and
$$
\int_0^{a^\star} i_0(a)da=0 \Rightarrow \int_0^{a^\star} i(t,a)da=0, \forall t \geq 0.
$$
Moreover if $\int_0^{a^\star} i_0(a)da=0$ then
$$
\int_0^{\infty} \beta(a)i(t,a)da=0, \forall t \geq 0
$$
and the solution is explicitly given by
\begin{equation}
i(t,a)=
\left\lbrace
\begin{array}{l}
e^{-\int_{a-t}^{a} \nu_I(\sigma)d\sigma} i_0(a-t), \text{{\rm if }} a-t \geq 0,\\
0, \text{{\rm if }} t-a \geq 0,
\end{array}
\right.
\end{equation}
therefore
$$
\lim_{t \to \infty} \Vert i(t,.) \Vert_{L^1}=0
$$
\end{theorem}
\begin{proof}Assume first that $S_0>0$. {In case where} $i_0 \in \partial \widehat{M}_0$. Then
$$
\int_0^{\infty}\Gamma_I^{+}(a)i_0(a) da=0.
$$
Moreover
$$
\int_0^{\infty}\Gamma_I^{+}(a)i(t,a) da \leq \int_0^{\infty}\Gamma_I^{+}(a)i^+(t,a) da =e^{\lambda^{+}t}\int_0^{\infty}\Gamma_I^{\pm}(a)i_0(a) da=0
$$
therefore
$$
\int_0^{\infty}\Gamma_I^{+}(a)i(t,a) da =0, \forall t \geq 0.
$$
Hence
$$
i(t,.) \in \partial \widehat{M}_0, \forall t \geq 0.
$$
{In case where $i_0 \in  \widehat{M}_0$. Then it follows that}
$$
\int_0^{\infty}\Gamma_I^{-}(a)i_0(a) da>0,
$$
and
$$
\int_0^{\infty}\Gamma_I^{-}(a)i(t,a) da \geq \int_0^{\infty}\Gamma_I^{-}(a)i^-(t,a) da =e^{\lambda^{-}t}\int_0^{\infty}\Gamma_I^{-}(a)i_0(a) da>0.
$$
Thus,
 $$
i(t,.) \in \widehat{M}_0, \forall t \geq 0.
$$
If $S_0=0$ and $i_0 \in \widehat{M}_0$ we can replace the initial time by any $t^\star>0$ small enough and we will have $S(t^\star)>0$ and $i(t^\star,.) \in  \widehat{M}_0$.
\end{proof}
\subsection{HIV infection model with infection age}
In this subsection, we will use {the} comparison principle to {determine} invariant sub-regions that can serve to deal with the asymptotic properties of {the following model presented in \cite{Huang-Takeuchi}:}
\begin{equation}\label{y.6}
\left\lbrace
\begin{array}{l}
T'(t)=s-d T(t)-k T(t)V(t),\\
\partial_t i(t,a)+\partial_a i(t,a)=-\delta(a)i(t,a), \text{{\rm for }} a \in (0,\infty)\\
i(t,0)=k T(t)V(t), \\
V'(t)=\int_0^{+\infty} p(a)i(t,a)da-cV(t)\\
T(0)=T_0\geq0,\ i(0,.)=i_0 \in \Li^1_+(\left(0,+\infty \right), \mathbb{R}), V(0)=V_0 \geq 0.
\end{array}
\right.
\end{equation}

\begin{assumption}\label{ASSy.11} We assume that the following conditions are satisfied
\begin{itemize}
\item[{\rm(i)}]$s>0$, $d>0$, $c>0$ and $k>0$;
\item[{\rm(ii)}]The function $a\rightarrow p(a)$ is bounded, uniformly
continuous from $[0,\infty)$ to $[0,\infty)$ and not identically zero;
\item[{\rm(iii)}]The function $ \delta(\cdot)\in \Li^{\infty }_+(\left(0,+\infty \right),\mathbb{R})$, and $\delta(a)\geq \delta_0$ for almost every $a\geq 0$ for some $\delta_0>0$.
\end{itemize}
\end{assumption}

\noindent \textbf{Volterra formulation:} Integrating the $i$-equation of {system \eqref{y.6} along} the characteristic lines gives
\begin{equation*}
i(t,a)=
\left\lbrace
\begin{array}{llll}
e^{-\int_{a-t}^{a} \delta(l)dl} i_0(a-t), & \text{if} &  a-t\geq 0,\\
e^{-\int_{0}^{a} \delta(l)dl} k T(t-a) V(t-a), & \text{if} & t-a> 0.
\end{array}
\right.
\end{equation*}
with $t\in \mathbb{R}_+ \mapsto (T(t),V(t),B(t))$ the unique continuous function satisfying for each $t\geq 0$ the following Volterra integral system
\begin{equation} \label{y.7}
\left\lbrace
\begin{array}{l}
\displaystyle T(t)=e^{-d t}T_0+\int_0^t e^{-d a} [s-k T(t-a)V(t-a)] da \\
\displaystyle V(t)=e^{-ct }V_0+\int_0^t e^{-ca} B(t-a)da\\
\displaystyle B(t)=\int_t^\infty p(a)  e^{-\int_{a-t}^{a} \delta(l)dl} i_0(a-t) da + \int_0^t  p(a) e^{-\int_{0}^{a} \delta(l)dl} B(t-a) da.
\end{array}
\right.
\end{equation}
\noindent \textbf{Integrated semigroup formulation:}
We will rewrite the $V$-equation and the $i$-equation as an abstract Cauchy problem. To do so we will consider $t\rightarrow T(t)$ as a known function. Set
$$
X=\mathbb{R} \times \mathbb{R}\times \Li^1(\left(0,+\infty \right),\mathbb{R})
$$
endowed with the usual product norm. Let $A:D(A)\subset X \to X$ be the linear operator defined by
$$
A
\left( \begin{array}{c}
V\\
\left(
\begin{array}{c}
0_{\mathbb{R}}\\
i
\end{array}
\right)
\end{array}\right)
=
\left( \begin{array}{c}
-cV\\
\left(
\begin{array}{c}
-i(0)\\
-i^\prime -\delta  i
\end{array}
\right)
\end{array}\right)
$$
and
$$
D(A)=\mathbb{R}\times \left\{ 0_{\mathbb{R}} \right\} \times \W^{1,1}(\left(0, +\infty \right), \mathbb{R}).
$$
Observe that
$$
\overline{D(A)}=\mathbb{R}\times\left\{ 0_{\mathbb{R}} \right\} \times \Li^{1}(\left(0, +\infty \right), \mathbb{R}) \varsubsetneq  X
$$
and  we set
$$
X_0:=\overline{D(A)}.
$$
To account the boundary condition and the non linearity we consider the map $F:\mathbb{R}\times X_0 \to X$ defined by
$$
F
\left(T,\left( \begin{array}{c}
V\\
\left(
\begin{array}{c}
0_{\mathbb{R}}\\
i
\end{array}
\right)
\end{array}\right)
\right)
=
\left( \begin{array}{c}
\int_0^{+\infty} p(a)i(a)da\\
\left(
\begin{array}{c}
k T V\\
0_{\Li^1}
\end{array}
\right)
\end{array}\right).
$$
Hence identifying $(V(t),i(t,.))$ and $(V_0,i_0)$ respectively with
$$
u(t):=
\left( \begin{array}{c}
V(t)\\
\left(
\begin{array}{c}
0_{\mathbb{R}}\\
i(,.)
\end{array}
\right)
\end{array}\right)\ \text{ and } \ u_0:=
\left( \begin{array}{c}
V_0\\
\left(
\begin{array}{c}
0_{\mathbb{R}}\\
i_0
\end{array}
\right)
\end{array}\right),
$$
we can rewrite the $V$-equation and the $i$-equation  in  \eqref{y.7}  as the following abstract Cauchy problem
\begin{equation*}
u^\prime(t)=Au(t)+F(T(t),u(t)), \text{ for } t \geq 0, \ u(0)=u_0 \in X_0.
\end{equation*}
\textbf{Boundedness property and invariant sub-regions:}
 Define
$$
I(t):=\int_0^{+\infty} i(t,a) da,\ \forall t>0 \ \text{ and } \ I(0)=I_0:=\int_{0}^{+\infty} i_0(a) da.
$$
Then by using the classical solutions we obtain that  the map $t\rightarrow (T(t),I(t),V(t))$ satisfies the following system
$$
\left\lbrace
\begin{array}{lll}
T'(t)=s-d T(t)-k T(t)V(t),\ t>0\\
I'(t)=k T(t)V(t)-\int_0^{+\infty}\delta(a)i(t,a)da,  \ t>0\\
V'(t)=\int_0^{+\infty} p(a)i(t,a)da-cV(t),\ t>0\\
T(0)=T_0\geq0,\ i(0,.)=i_0 \in \Li^1_+(\left(0,+\infty \right), \mathbb{R}), V(0)=V_0 \geq 0.
\end{array}
\right.
$$
By {adding up the equations of $T$ and $I$, we get}
$$
(T(t)+I(t))'=s-d T(t)-\int_0^{+\infty}\delta(a)i(t,a)da\leq s-d T(t)-\delta_0 I(t),\ \forall t>0.
$$
Hence by setting
$$
d_0:=\min(d,\delta_0)
$$
we obtain for each $t\geq 0$  that
$$
T(t)+I(t) \leq e^{-d_0 t}(T_0+I_0)+\int_0^t e^{-d_0 (t-a)} s da=e^{-d_0 t}(T_0+I_0)+\dfrac{s}{d_0} \left(1-e^{-d_0 t}\right)
$$
which implies that
\begin{equation}\label{y.8}
T(t)+I(t)\leq \max\left(T_0+I_0,\dfrac{s}{d_0}\right),\ \forall t\geq 0.
\end{equation}
{In view of \eqref{y.8}, it follows from the $V$-equation that}
$$
V'(t)\leq \Vert p \Vert_{\infty} \max\left(T_0+I_0,\dfrac{s}{d_0}\right) -cV(t),\ \forall t>0,\ V(0)=V_0.
$$
{By the same arguments as above, we have}
\begin{equation}\label{y.9}
V(t)\leq \max\left\lbrace V_0, \dfrac{\Vert p \Vert_{\infty}}{c}\max\left(T_0+I_0,\dfrac{s}{d_0}\right) \right\rbrace,\ \forall t\geq 0.
\end{equation}
{Since the set of initial conditions giving classical solutions is dense in $L^1$, we see} that \eqref{y.8} and \eqref{y.9} still hold true for any given non negative initial conditions of \eqref{y.6}. This will allow us to prove the next result.
\begin{lemma}\label{PROPy.12a} Assume that the initial value $T(0)=T_0>0$. Then there exists $T_{+}>T_{-}>0$ such that
$$T_{-} \leq T(t)\leq T_{+},\ \mbox{for}\ t\geq 0.$$
\end{lemma}
\begin{proof}
The upper bound of $t\rightarrow T(t)$ follows from \eqref{y.8}. {We are in a position to determine the lower bound of $t\rightarrow T(t)$. From \eqref{y.9}, one has}
$$
0\leq V(t) \leq \dfrac{\Vert p \Vert_{\infty}}{c}\max\left(T_0+I_0,\dfrac{s}{d_0}\right),\ \forall t\geq 0.
$$
{By using the $T$-equation in \eqref{y.6}, and setting}
$$
d_1:=d+k \dfrac{\Vert p \Vert_{\infty}}{c}\max\left(T_0+I_0,\dfrac{s}{d_0}\right),
$$
we get
$$
T'(t)\geq s-d_1T(t),\ t>0,\ T(0)=T_0>0.
$$
Hence
$$
T(t)\geq e^{-d_1 t} T_0+\int_0^t e^{-d_1 (t-a)} s da=e^{-d_1 t} T_0+\dfrac{s}{d_1}\left( 1-e^{-d_1 t}\right),\ \forall t\geq 0.
$$
{Then} the result follows by setting $T_-:=\min\left( T_0, \dfrac{s}{d_1}\right)$.
\end{proof}

As a consequence of  Lemma \ref{PROPy.12a} and the comparison principle, we have the following result:
\begin{theorem}\label{THy.13} Assume that $i_{\pm}(t,a)$ satisfies
\begin{equation}\label{y.10}
\left\lbrace
\begin{array}{l}
\partial_t i_{\pm}(t,a)+\partial_a i_{\pm}(t,a)=-\delta(a)i_{\pm}(t,a), \text{ for } a \geq 0,\ t\geq0,\\
i_{\pm}(t,0)=k T_{\pm}V_\pm (t), \\
V_\pm'(t)=\int_0^{+\infty} p(a)i_\pm (t,a)da-cV_\pm (t)\\
i_{\pm}(0,.)=i_0 \in \Li^1_+(\left(0,+\infty \right), \mathbb{R}),\ V_\pm(0)=V_0.
\end{array}
\right.
\end{equation}
Then
\begin{equation}
i_{-}(t,a)\leq i(t,a) \leq i_{+}(t,a),\ \mbox{for}\ t\geq0, \ \mbox{for a. e.}\ a\geq 0
\end{equation}
and
\begin{equation}
V_-(t)\leq V(t) \leq V_+(t),\ \forall t\geq 0.
\end{equation}
\end{theorem}

The foregoing Theorem \ref{THy.13} will allow us to determine the invariant sub-regions that can serve to {study the uniform persistence of  system \eqref{y.6}.} Let
$$
\Gamma^{\pm}_I(a)=\dfrac{kT_\pm}{\lambda^\pm +c} \int_a^{+\infty} e^{-\int_a^\theta [\delta(l)+\lambda^\pm]} p(\theta) d\theta,\ \forall a\geq 0,
$$
where $\lambda^{\pm}$ is chosen to satisfy $\Gamma_I^{\pm}(0)=1$. That is to say that $\lambda^{\pm} \in \mathbb{R}$ satisfies
$$
\dfrac{kT_\pm}{\lambda^\pm +c} \int_0^{+\infty} e^{-\int_0^\theta [\delta(l)+\lambda^\pm]} p(\theta) d\theta=1.
$$
Next we define the bounded linear operators $\Gamma^\pm : \mathbb{R} \times \Li^1((0,+\infty),\mathbb{R}) \rightarrow  \mathbb{R}$ by
\begin{equation}\label{y.13}
\Gamma^\pm\left(\begin{array}{cc}
V\\
i
\end{array}\right)=\dfrac{kT_\pm}{\lambda^\pm +c}V+ \int_0^{+\infty} \Gamma_I^\pm(a) i(a) da.
\end{equation}
\begin{lemma} For each $t \geq 0$,
$$
\Gamma^\pm\left(\begin{array}{cc}
V_\pm(t)\\
i_\pm(t,\cdot)
\end{array}\right)=e^{\lambda^\pm t} \Gamma^\pm\left(\begin{array}{cc}
V_0\\
i_0
\end{array}\right).
$$
\end{lemma}
\begin{proof}
The function $ a \to \Gamma_I^{\pm}(a)$ satisfies
\begin{equation*}
\left\lbrace
\begin{array}{ll}
(\Gamma_I^{\pm})'(a)=[\delta(a)+\lambda^{\pm}]\Gamma_I^{\pm}(a)-\dfrac{k T_{\pm}}{\lambda^\pm+c} p(a),\ \mbox{for a.e.}\ a\geq 0,\\
\Gamma_I^{\pm}(0)=1.
\end{array}
\right.
\end{equation*}
Hence by using classical solution of \eqref{y.10}, { for each $t>0$, we have}
$$
\begin{array}{lllll}
 \displaystyle \frac{d}{d t}\Gamma^\pm\left(\begin{array}{cc}
V_\pm(t)\\
i_\pm(t,\cdot)
\end{array}\right)&=& \displaystyle \dfrac{kT_\pm}{\lambda^\pm+c}\left[\int_0^{+\infty} p(a) i_\pm(t,a)da-cV_\pm(t)\right]\\
& &  \displaystyle +\int_0^{+\infty} \Gamma^\pm_I(a)\left[-\partial_ai(t,a)-\delta(a) i(t,a)\right]da\\
&=& \displaystyle \dfrac{kT_\pm}{\lambda^\pm+c}\left[\int_0^{+\infty} p(a) i_\pm(t,a)da-cV_\pm(t)\right]\\
&  &  \displaystyle +\Gamma^\pm_I(0)i(t,0) +\int_0^{+\infty} (\Gamma^\pm_I)'(a)i(t,a)da\\
& & \displaystyle  -\int_0^{+\infty} \Gamma^\pm_I(a) \delta(a) i(t,a)da\\
&=& \displaystyle -\dfrac{kT_\pm}{\lambda^\pm+c}cV_\pm(t)\\
&  &  \displaystyle +\Gamma^\pm_I(0)i(t,0) +\lambda^\pm \int_0^{+\infty} \Gamma^\pm_I(a)i(t,a)da.
\end{array}
$$
Recalling that $\Gamma^\pm_I(0)=1$ and $i_\pm(t,0)=kT_\pm V_\pm(t)$, we obtain that for each $t\geq 0$
$$
\frac{d}{d t}\Gamma^\pm\left(\begin{array}{cc}
V_\pm(t)\\
i_\pm(t,\cdot)
\end{array}\right)=\lambda^\pm \dfrac{kT_\pm}{\lambda^\pm+c} V(t)+\lambda^\pm \int_0^{+\infty} \Gamma^\pm_I(a)i(t,a)da=\lambda^\pm \Gamma^\pm\left(\begin{array}{cc}
V_\pm(t)\\
i_\pm(t,\cdot)
\end{array}\right)
$$
which implies that
$$
\Gamma^\pm\left(\begin{array}{cc}
V_\pm(t)\\
i_\pm(t,\cdot)
\end{array}\right)=e^{\lambda^\pm t}\Gamma^\pm \left(\begin{array}{cc}
V_\pm(0)\\
i_\pm(0,\cdot)
\end{array}\right),\ \forall t\geq 0.
$$
The result follows by using the density of the initial conditions giving classical solution combined together with the fact that $\Gamma^\pm$ is a bounded linear operator.
\end{proof}

In order to deal with the persistence property of system \eqref{y.6}  we need to assume {that $p \not\equiv 0$ and} let
$$
a^\star:=\sup \left\{a>0: \int_a^{\infty}p(\sigma) e^{-\sigma} d\sigma >0  \right\} \in (0, \infty].
$$
Define the interior sub-domain
$$
\widehat{M}_0=\left \{\left(\begin{array}{c}
V\\
i
\end{array}\right) \in \mathbb{R}_+\times \Li^1_+((0,+\infty),\mathbb{R}): V+\int_0^{a^\star} i(a)da>0 \right\},
$$
and the boundary sub-domain
$$
\partial \widehat{M}_0=\left \{\left(\begin{array}{c}
V\\
i
\end{array}\right) \in \mathbb{R}\times \Li^1_+((0,+\infty),\mathbb{R}): V+\int_0^{a^\star} i(a)da=0 \right\}.
$$
\begin{theorem} The domains $\left[ 0,\infty \right) \times \widehat{M}_0 $ and  $\left[ 0,\infty \right) \times  \partial \widehat{M}_0 $  are positively invariant by the semiflow generated by \eqref{y.6}. {That is to say} that
$$
V_0+\int_0^{a^\star} i_0(a)da>0 \Rightarrow V(t)+\int_0^{a^\star} i(t,a)da>0, \forall t \geq 0
$$
and
$$
V_0+\int_0^{a^\star} i_0(a)da=0 \Rightarrow V(t)+ \int_0^{a^\star} i(t,a)da=0, \forall t \geq 0.
$$
Moreover if $V_0+\int_0^{a^\star} i_0(a)da=0$ then
$$
V(t)=0, \forall t \geq 0
$$
and the {solution of the} $i$-equation is explicitly given by
\begin{equation}
i(t,a)=
\left\lbrace
\begin{array}{l}
e^{-\int_{a-t}^{a} \delta(l)di} i_0(a-t), \text{{\rm if }} a-t \geq 0,\\
0, \text{{\rm if }} t-a \geq 0.
\end{array}
\right.
\end{equation}
Therefore
$$
\lim_{t \to \infty} \Vert i(t,.) \Vert_{L^1}=0.
$$
\end{theorem}
\begin{proof} We first prove that $[0,+\infty)\times \partial \widehat{M}_0$ is positively invariant.  Let  $\left(\begin{array}{c}
V_0\\
i_0
\end{array}\right) \in \partial \widehat{M}_0$. Then we have
$$
V_0=0 \ \text{ and } i_0(a)=0,\  \mbox{for a.e.} \ a\in (0,a^\star)
$$
so that
$$
\Gamma^{+}\left(\begin{array}{c}
V_0\\
i_0
\end{array}\right)=0.
$$
By using  Theorem \ref{THy.13} and the definition of $\Gamma^\pm$ in \eqref{y.13} we have
$$
0\leq \Gamma^{+}\left(\begin{array}{c}
V(t)\\
i(t,.)
\end{array}\right)\leq \Gamma^{+}\left(\begin{array}{c}
V_+(t)\\
i_+(t,.)
\end{array}\right)=e^{\lambda^+ t} \Gamma^{+}\left(\begin{array}{c}
V_0\\
i_0
\end{array}\right)=0.
$$
{Then} we deduce, for each $t\geq 0$, that
$$
\dfrac{kT_+}{\lambda^+ +c}V(t)+ \int_0^{+\infty} \Gamma_I^+(a) i(t,a) da=0.
$$
Hence
$$
\left(\begin{array}{c}
V(t)\\
i(t,.)
\end{array}\right)\in \partial \widehat{M}_0,\ \forall t\geq 0.
$$
Next we prove that $[0,+\infty)\times \widehat{M}_0$ is positively invariant.  Let  $\left(\begin{array}{c}
V_0\\
i_0
\end{array}\right) \in \widehat{M}_0$ be given.
If  $T_0>0$ then we have
$$
\Gamma^{-}\left(\begin{array}{c}
V_0\\
i_0
\end{array}\right)>0.
$$
Combining  Theorem \ref{THy.13} and the definition of $\Gamma^\pm$ in \eqref{y.13}  we obtain
$$
0< e^{\lambda^- t} \Gamma^{-}\left(\begin{array}{c}
V_0\\
i_0
\end{array}\right)=\Gamma^{-}\left(\begin{array}{c}
V_-(t)\\
i_-(t,.)
\end{array}\right)\leq  \Gamma^{-}\left(\begin{array}{c}
V(t)\\
i(t,.)
\end{array}\right),\ \forall t\geq 0.
$$
Hence
$$
0<\dfrac{kT_-}{\lambda^- +c}V(t)+ \int_0^{+\infty} \Gamma_I^-(a) i(t,a) da,\ t\geq 0,
$$
and we deduce that
$$
\left(\begin{array}{c}
V(t)\\
i(t,.)
\end{array}\right)\in \widehat{M}_0,\ \forall t\geq 0.
$$
If $T_0=0$  we can replace the initial time by any $t^\star>0$ small enough and we will have $T(t^\star)>0$ and $\left(\begin{array}{c}
V(t^\star)\\
i(t^\star,.)
\end{array}\right)\in \widehat{M}_0$.
\end{proof}
\subsection{Age-Structured Population Dynamics Models}
Let $p \in [1, +\infty)$ and $q \in (1, +\infty]$ with $1/p+1/q=1$.
In this subsection we consider the following class of age structured model
\begin{equation} \label{6.2}
\left\lbrace
\begin{array}{l}
\partial_t u(t,a)+\partial_a u(t,a)=\mu(\G(t),a)u(t,a), \text{ for } a \geq 0\\
u(t,0)=\int_0^{+\infty} \beta(\Sigma(t),a)u(t,a)da \\
u(0,.)=u_0 \in \Li^p_+(\left(0,+\infty \right), \mathbb{R}^n)
\end{array}
\right.
\end{equation}
with
$$
\G(t):=\int_0^{+\infty}  \alpha(a)u(t,a)da \in \mathbb{R}^n
$$
and
$$
\Sigma(t):=\int_0^{+\infty}  \sigma(a)u(t,a)da \in \mathbb{R}^n.
$$

\noindent \textbf{Abstract Cauchy problem reformulation: }
Set
$$
 { X=\mathbb{R}^n \times \Li^p(\left(0,+\infty \right),\mathbb{R}^n)}
$$
endowed with the usual product norm.

Consider the linear operator $A:D(A)\subset X \to X$
$$
A
\left(
\begin{array}{c}
0_{\mathbb{R}^n}\\
\varphi
\end{array}
\right)
=
\left(
\begin{array}{c}
-\varphi(0)\\
-\varphi^\prime
\end{array}
\right)
$$
and
$$
D(A)=\left\{ 0_{\mathbb{R}^n} \right\} \times \W^{1,p}(\left(0, +\infty \right), \mathbb{R}^n).
$$
Then the closure of the domain of $A$ is
$$
X_0:=\overline{D(A)}=\left\{ 0_{\mathbb{R}^n} \right\} \times \Li^{p}(\left(0, +\infty \right), \mathbb{R}^n).
$$
Consider $F:X_0 \to X$ given by
$$
F
\left(
\begin{array}{c}
0_{\mathbb{R}^n}\\
\varphi
\end{array}
\right)
=
\left(
\begin{array}{c}
\C(\varphi)\\
\D(\varphi)
\end{array}
\right)
$$
 where we have set
$$
\C(\varphi):=\int_0^{+\infty} \beta(\int_0^{+\infty}  \sigma(r)\varphi(r)dr,a)\varphi(a)da,\ \varphi\in \Li^{p}(\left(0, +\infty \right), \mathbb{R}^n)
$$
and
$$
\D(\varphi)(a)=\mu(\int_0^{+\infty}  \alpha(r)\varphi(r)dr,a)\varphi(a),\ \varphi\in \Li^{p}(\left(0, +\infty \right), \mathbb{R}^n).
$$
By identifying $u(t,.)$ with $v(t):=
\left(
\begin{array}{c}
0_{\mathbb{R}^n}\\
u(t,.)
\end{array}
\right)$ we can rewrite the partial differential equation \eqref{6.2} as the following abstract Cauchy problem
\begin{equation*}
v^\prime(t)=Av(t)+F(v(t)), \text{ for } t \geq 0, v(0)=\left(
\begin{array}{c}
0_{\mathbb{R}^n}\\
u_0
\end{array}
\right) \in X_0.
\end{equation*}
Age structured problem in $\Li^p$ have been studied specifically in Magal and Ruan \cite{Magal-Ruan07} and Assumptions \ref{ASS2.2} and \ref{ASS2.3} are satisfied for this specific class of examples.

{Set
$$
X_+=\mathbb{R}_+^n \times \Li_+^p(\left(0,+\infty \right),
$$
and we define} $B:X_0 \to X$ by
$$
B
\left(
\begin{array}{c}
0_{\mathbb{R}^n}\\
\varphi
\end{array}
\right)
=
\left(
\begin{array}{c}
0_{\mathbb{R}^n}\\
\varphi
\end{array}
\right).
$$
The next results show that the linear operator $A-\gamma B$ is resolvent positive in the sense that
$$
\left(\lambda I-(A- \gamma B) \right)^{-1} X_+ \subset X_+.
$$
That is to say Assumption~\ref{ASS3.0}-(i) and Assumption~\ref{ASS3.1}-(i) {are} satisfied.
\begin{lemma} \label{LE6.11}
For each $\lambda >-\gamma$, $\lambda$ belongs to the resolvent set of $\rho(A-\gamma B)$, and we have the following explicit formula for the resolvent of $A-\gamma B$:
$$
\left(\lambda I-(A- \gamma B) \right)^{-1}
\left(
\begin{array}{c}
\alpha\\
\psi
\end{array}
\right)
=
\left(
\begin{array}{c}
0_{\mathbb{R}^n}\\
\varphi
\end{array}
\right),
$$
where
$$\varphi(a)=e^{-(\lambda +\gamma)a}\alpha+\int_0^a e^{-(\lambda +\gamma)(a-l)}\psi(l)dl,\ \text{ for almost every } a\in [0,+\infty).$$
\end{lemma}
\begin{assumption}[{Positivity}]\label{ASS6.9.0} We assume that the following {conditions} are satisfied
\begin{itemize}
\item[{\rm(i)}]
$$
\alpha, \sigma \in \Li^{q}_+(\left(0,+\infty \right), \M_n(\mathbb{R}) ),
$$
$$
\mu \in \C^1 \left(  \mathbb{R}^n, \Li^\infty(\left(0,+\infty \right), \M_n(\mathbb{R}) {)} \right),
$$
$$
\beta \in \C^1 \left(  \mathbb{R}^n, \Li^q(\left(0,+\infty \right), \M_n(\mathbb{R}){)} \right);
$$

\item[{\rm(ii)}]the birth function $\C:\Li^p(\left(0,+\infty \right), \mathbb{R}^n) \to \mathbb{R}^n$
$$
\C({\varphi}):=\int_0^{+\infty} \beta(\int_0^{+\infty}  \sigma(r){\varphi}(r)dr,a){\varphi}(a)da
$$
is non-negative (i.e. $\C(\Li^p_+) \subset \mathbb{R}^n_+$).
\item[{\rm(iii)}] Consider $\D: \Li^p(\left(0,+\infty \right), \mathbb{R}^n) \to \Li^p(\left(0,+\infty \right), \mathbb{R}^n)$ the map defined by
$$
\D({\varphi})(a)=\mu(\int_0^{+\infty}  \alpha(r){\varphi}(r)dr,a){\varphi}(a).
$$
We assume that for each $\rm{M}>0$ there exists $\gamma=\gamma(\rm{M})>0$ such that the map ${\varphi} \to \gamma {\varphi}+\D({\varphi})$ is non-negative on $ \left\{{\varphi} \in \Li^p_+ : \Vert {\varphi} \Vert_{L^p} \leq M \right\} $.
\end{itemize}
\end{assumption}
The first part of Assumption \ref{ASS3.0} is now clearly satisfied and the property (ii) of Assumption \ref{ASS3.0} {can be readily} checked. {In fact, by using Assumptpion \ref{ASS6.9.0}, we} deduce that for $M>0$ we can find $\gamma>0$ such that
$$ (F+\gamma B)
\left(
\begin{array}{c}
0_{\mathbb{R}^n}\\
\varphi
\end{array}
\right)
\geq 0
$$
whenever $\Vert \varphi \Vert_{L^p} \leq M$.

Therefore, by Theorem \ref{TH3.1}, we obtain the positivity of the semiflow.

\begin{theorem}[Positive Semiflow] ~\\  \label{TH6.11}
Let Assumption \ref{ASS6.9.0} be satisfied. Then the semiflow generated by the age structured problem \eqref{6.2} is non negative on $\Li^p_+(\left(0,+\infty \right), \mathbb{R}^n)$. That is to say if $u_\varphi(t,a)$ is the solution of system \eqref{6.2} with initial value $\varphi$ then
$$
 \varphi \geq 0 \Rightarrow u_\varphi(t,a) \geq 0 , \forall t \in \left[0, \tau_\varphi \right),
$$
where $\tau_\varphi$ is the maximal time of existence of the solution $u_\varphi(t,a)$.

\end{theorem}

\begin{assumption}[{Monotonicity}]\label{ASS6.9} We assume that the following {conditions} are satisfied
\begin{itemize}
\item[{\rm(i)}]
$$
\alpha, \sigma \in \Li^{q}_+(\left(0,+\infty \right), \M_n(\mathbb{R}) ),
$$
$$
\mu \in \C^1 \left(  \mathbb{R}^n, \Li^\infty(\left(0,+\infty \right), \M_n(\mathbb{R}) {)}\right),
$$
$$
\beta \in \C^1 \left(  \mathbb{R}^n, \Li^q(\left(0,+\infty \right), \M_n(\mathbb{R}) {)}\right);
$$

\item[{\rm(ii)}]the birth function $\C:\Li^p(\left(0,+\infty \right), \mathbb{R}^n) \to \mathbb{R}^n$
$$
\C({\varphi}):=\int_0^{+\infty} \beta(\int_0^{+\infty}  \sigma(r){\varphi}(r)dr,a){\varphi}(a)da
$$
is non-negative (i.e. $\C(\Li^p_+) \subset \mathbb{R}^n_+$) and is a monotone increasing function on $\Li^p_+$.
\item[{\rm(iii)}] Consider $\D: \Li^p(\left(0,+\infty \right), \mathbb{R}^n) \to \Li^p(\left(0,+\infty \right), \mathbb{R}^n)$ the map defined by
$$
\D({\varphi})(a)=\mu(\int_0^{+\infty}  \alpha(r){\varphi}(r)dr,a){\varphi}(a).
$$
We assume that for each $\rm{M}>0$ there exists $\gamma=\gamma(\rm{M})>0$ such that the map ${\varphi} \to \gamma {\varphi}+\D({\varphi})$ is non-negative and monotone increasing on $ \left\{{\varphi} \in \Li^p_+ : \Vert {\varphi} \Vert_{L^p} \leq M \right\} $.
\end{itemize}
\end{assumption}
The first part of Assumption \ref{ASS3.1} is now clearly satisfied by Lemma \ref{LE6.11} and the property (ii) of Assumption \ref{ASS3.1} {can be readily} checked. Indeed, by using Assumptpion \ref{ASS6.9} we deduce that for $M>0$ we can find $\gamma>0$ such that
$$\varphi \geq \psi  \Rightarrow (F+\gamma B)
\left(
\begin{array}{c}
0_{\mathbb{R}^n}\\
\varphi
\end{array}
\right)
\geq
(F+\gamma B)
\left(
\begin{array}{c}
0_{\mathbb{R}^n}\\
\psi
\end{array}
\right)
\geq 0
$$
whenever $\Vert \varphi \Vert_{L^p} \leq M$  and $\Vert \psi \Vert_{L^p} \leq M$.

By applying Theorem \ref{TH3.4} we obtain the following result.
\begin{theorem}[Monotone Semiflow] ~\\  \label{TH6.15}
Let Assumption \ref{ASS6.9} be satisfied. Then the semiflow generated by the age structured problem \eqref{6.2} is non negative and monotone increasing on $\Li^p_+(\left(0,+\infty \right), \mathbb{R}^n)$. That is to say if $u_\varphi(t,a)$ (respectively $u_\phi(t,a)$) is the solution of system \eqref{6.2} with initial value $\varphi$ (respectively $\phi$) then
$$
0 \leq \varphi \leq \phi \Rightarrow 0 \leq u_\varphi(t,a) \leq u_\phi(t,a), \forall t \in \left[0, \min(\tau_\varphi,\tau_\phi) \right),
$$
where $\tau_\varphi$ (respectively $\tau_\phi$) is the maximal time of existence of the solution $u_\varphi(t,a)$ (respectively $u_\phi(t,a)$).
\end{theorem}
\textbf{Comparison principle:} Let $\varphi \in \W^{1,p}(\left(0,+\infty \right), \mathbb{R}^n)$ with $\varphi \geq 0$.
Let $w\in C^{1} ([0,\tau], \Li^p(\left(0,+\infty \right), \mathbb{R}^n) ) \cap C([0,\tau], \W^{1,p}(\left(0,+\infty \right), \mathbb{R}^n) )$ such that
$$
w(0,a)=\varphi(a), \text{ for almost every } a\geq 0.
$$
Then by setting
$$
v(t)=
\left(
\begin{array}{l}
0_{\mathbb{R}^n}\\
w (t,.)
\end{array}
\right)
$$
we have
$$
v'(t) \leq Av(t)+ F(v(t)), \ \forall t \in [0,\tau],
$$
{which} is equivalent to
$$
\left(
\begin{array}{c}
0_{\mathbb{R}^n}\\
\dfrac{\partial w (t,.)}{\partial t}
\end{array}
\right)
\leq
\left(
\begin{array}{c}
-w (t,0)\\
-\dfrac{\partial w (t,.)}{\partial a}
\end{array}
\right)
+
\left(
\begin{array}{c}
\C \left( w  (t,.)\right)\\
\D \left( w (t,.)\right)
\end{array}
\right), \ \forall t \in [0,\tau].
$$
By using Proposition  \ref{PROP4.3} and Proposition \ref{PROP4.4} we obtain the following result.

\begin{theorem}[Comparison Principle]~\\ \label{TH6.12} Let Assumption \ref{ASS6.9} be satisfied, $\varphi \in \W^{1,p}(\left(0,+\infty \right), \mathbb{R}^n)$ with $\varphi \geq 0$, and
$u_\varphi(t,a)$ be the solution of system \eqref{6.2} with initial value $\varphi$. Let $$
w \in C^{1} ([0,\tau], \Li^p(\left(0,+\infty \right), \mathbb{R}^n) ) \cap C([0,\tau], \W^{1,p}(\left(0,+\infty \right), \mathbb{R}^n) ).
$$

 \begin{itemize}
\item[{\rm(i)}] Assume in addition that $w(t,a)$ satisfies the following inequality for each $ t \in [0,\tau]$,
\begin{equation*}
\left\lbrace
\begin{array}{l}
\dfrac{\partial w(t,a)}{\partial t}+\dfrac{\partial w(t,a)}{\partial a}\leq \D \left( w(t,.)\right)(a)\\
w(t,0)\leq \C\left(w(t,.)\right),
\end{array}
\right.
\end{equation*}
$$
w(t,.) \geq 0, \forall t \in [0, \tau],
$$
and
$$
w(0,.) \leq \varphi.
$$
{Then}
$$
w(t,.)\leq u_\varphi(t,.), \ \forall t \in [0,\tau].
$$
\item[{\rm(ii)}] Assume in addition that $w(t,a)$ satisfies the following inequality
\begin{equation*}
\left\lbrace
\begin{array}{l}
\dfrac{\partial w(t,a)}{\partial t}+\dfrac{\partial w(t,a)}{\partial a}\geq \D\left(w(t,.)\right)(a)\\
w(t,0)\geq \C\left(w(t,.)\right), \ \forall t \in [0,\tau].
\end{array}
\right.
\end{equation*}
$$
w(t,.) \geq 0, \forall t \in [0, \tau],
$$
and
$$
w(0,.) \geq \varphi.
$$
Then
$$
w(t,.)\geq u_\varphi(t,.), \forall t \in [0,\tau].
$$
\end{itemize}
\end{theorem}
\textbf{Increasing and Decreasing solutions:} Let $\varphi \in \W^{1,p}(\left(0,+\infty \right), \mathbb{R}^n)$ with $\varphi \geq 0$. Then
$$
x=
\left(
\begin{array}{c}
0_{\mathbb{R}^n} \\
\varphi
\end{array}
\right)
\in D(A),
$$
and we have
$$
Ax+F(x) \geq 0 \Leftrightarrow
\left(
\begin{array}{c}
-\varphi(0)\\
-\varphi^\prime
\end{array}
\right)
+
\left(
\begin{array}{c}
\C\left(\varphi\right)\\
\D\left(\varphi\right)
\end{array}
\right) \geq 0
$$
therefore by applying Theorem \ref{TH5.1}  we obtain the following result.

\begin{theorem}[Increasing and Decreasing Solutions] ~\\ \label{TH6.13}
Let Assumption \ref{ASS6.9} be satisfied; $\varphi \in \W^{1,p}(\left(0,+\infty \right), \mathbb{R}^n)$ with $\varphi \geq 0$, and $u_\varphi(t,a)$ be the solution of system \eqref{6.2} with initial value $\varphi$. Then we have the following alternative
 \begin{itemize}
\item[{\rm(i)}] Assume in addition that $\varphi^\prime \leq \D\left(\varphi\right)$ and $\varphi(0) \leq \C\left(\varphi\right)$ then $t \to u_\varphi(t,a)$ is increasing.
 \item[{\rm(ii)}] Assume in addition that $\varphi^\prime \geq \D\left(\varphi\right)$ and $\varphi(0) \geq \C\left(\varphi\right)$ then $t \to u_\varphi(t,a)$ is decreasing.
 \end{itemize}
\end{theorem}
\begin{example}\label{EX6.3} Assume that $n=1$ and assume for simplicity that
$$
\mu(\G,a)=-\mu_0(a)\;  \chi(\G),
$$
with
$$
\mu_0 \in \Li^{\infty}_+(\left(0,+\infty \right), \mathbb{R}), \text{ and } h(\mathbb{R}_+) \subset \mathbb{R}_+.
$$
Then the map
$$
(\gamma I+\D)(u)(a)= \left[\gamma -\mu_0(a) \chi(\int_0^{+\infty}  \alpha(r)u(r)dr)\right]u(a)
$$
will {be} monotone increasing for $\gamma>0$ large enough if the map $\chi$  is decreasing on $[0, +\infty)$. Indeed whenever $\chi$ is decreasing on $[0, +\infty)$ the condition (iii) in {Assumption~\ref{ASS6.9.0}
and Assumption~\ref{ASS6.9} will} be satisfied for each $\gamma> \Vert \mu_0 \Vert_{L^\infty}\chi(0)$.
\end{example}

\begin{remark} To clarify the presentation,  we did not {consider the non-autonomous case in this subsection.} But our results can be {applied, for example, to the} following non-autonomous age structured model
\begin{equation*}
\left\lbrace
\begin{array}{l}
\partial_t u(t,a)+\partial_a u(t,a)=\mu(\G(t) , t ,a)u(t,a), \text{{\rm for }} a \in (0,\infty) \\
u(t,0)=\int_0^{+\infty} \beta(\Sigma(t), t ,a)u(t,a)da \\
\frac{d \V(t)}{dt}= \F(\V(t),\G(t))\\
u(0,.)=u_0 \in \Li^p_+(\left(0,+\infty \right), \mathbb{R}^n) \text{ and }V(0)=V_0 \in \mathbb{R}^n
\end{array}
\right.
\end{equation*}
where the quantities $\G(t)$ and $\Sigma(t)$ are defined as above.
\end{remark}

\end{document}